\documentclass[a4paper,11pt]{article}

\usepackage[affil-it]{authblk}
\usepackage{colortbl}
\usepackage{multirow}%
\usepackage{fullpage,marginnote}
\usepackage{amsmath,amsthm,amsfonts,amssymb,mathrsfs}
\usepackage{graphicx}
\usepackage{bm}
\usepackage{natbib}
\usepackage[colorlinks=true,linkcolor=black,citecolor=blue]{hyperref}
\usepackage[titletoc,title]{appendix}
\RequirePackage[OT1]{fontenc}
\usepackage{graphicx,psfrag,epsf}
\usepackage{enumerate}
\usepackage{booktabs,multirow} 
\usepackage{url} 
\usepackage{dsfont}
\usepackage{enumitem}

\newtheorem{theorem}{Theorem}
\newtheorem{definition}{Definition}
\newtheorem{corollary}[theorem]{Corollary}

\newtheorem{lemma}[theorem]{Lemma}

\newcommand{\diff}{\mathrm{d}}
\renewcommand{\le}{\leqslant}
\renewcommand{\ge}{\geqslant}
\renewcommand{\leq}{\leqslant}
\renewcommand{\geq}{\geqslant}
\renewcommand{\tilde}{\widetilde}

\begin{document}

\title{Safe adaptive importance sampling: a mixture approach}
\author{Bernard Delyon and Fran\c cois Portier\thanks{Corresponding author, \texttt{francois.portier@gmail.com}.}}
\affil{Universit\'e de Rennes 1 and Institut Polytechnique de Paris}
\maketitle

\begin{abstract}
This paper investigates \textit{adaptive importance sampling} algorithms 
for which the \textit{policy}, 
the sequence of distributions used to generate the particles, 
is a mixture distribution between a flexible \textit{kernel density estimate} 
(based on the previous particles), and a ``safe'' heavy-tailed density. 
When the share of samples generated according to the safe density goes to zero 
but not too quickly, two results are established: 
(i) uniform convergence rates are derived for the policy toward the target density; 
(ii) a central limit theorem is obtained for the resulting integral estimates. 
The fact that the asymptotic variance is the same as the variance of an ``oracle'' procedure with variance-optimal policy, illustrates the benefits of the approach. 
In addition, a subsampling step (among the particles) can be conducted before constructing the kernel estimate in order to decrease the computational effort without altering the performance of the method.
The practical behavior of the algorithms is illustrated in a simulation study.

\bigskip


\end{abstract}

\section{Introduction}

The Monte Carlo simulation framework has become indisputably fruitful for exploring probability density functions, 
especially when the ambient space has a large dimension. 
Domains of application include for instance computational physics, Bayesian modeling and optimization. Among the most popular Monte Carlo approaches, there are \textit{Markov Chains Monte Carlo}, \textit{sequential Monte Carlo} and \textit{adaptive importance sampling} (AIS). Reference textbooks includes  \cite{evans:2000}, \cite{robert:2004}, \cite{del:2013}, \cite{owen:13}.
%

This study is part of the AIS methodology, the main characteristic of which is to alternate between the two following stages: (i) new particles are generated under a certain probability distribution called the \textit{policy} $q_k$, and (ii) the next policy $q_{k+1}$ is settled on using the new particles. This last point reflects the {\it adaptive} character of the method. Classically, 
two families of methods can be distinguished depending on the approach taken to model the policy: 
\textit{parametric} and \textit{nonparametric}. 

Pioneer works on adaptive schemes have focused on parametric families to model the policy. 
They include, among others \cite{kloek+v:1978}, \cite{geweke:1989}, \cite{ho+b:1992}, \cite{owen+z:2000}, 
\cite{cappe+g+m+r:2004}, \cite{cappe+d+g+m+r:2008} (see also \cite{elvira+m+l+b:2015} for a review on the variant called \textit{adaptive multiple importance sampling}). 
In \cite{ho+b:1992}, martingale techniques were successfully employed to describe AIS schemes and their approach was recently extended \citep{delyon+p:2018} to obtain a central limit theorem for AIS integral estimates when $q_k$ is chosen out of a parametric family. 
Nonparametric approaches were originally based on kernel smoothing techniques 
and include \cite{west:1993}, \cite{givens+r:1996}, \cite{zhang:1996}, \cite{neddermeyer:2009}. All these authors defined the policy as a kernel density estimate based on the previous particles re-weighted by importance weights.

 In the present work, the policy $q_k$ is designed to estimate a certain density function $f$, called the \textit{target}. This is convenient when many integrals $\int gf$ are to be computed \citep[section 1.2]{ho+b:1992}, making less efficient to use any criterion that would depend on $g$ as for instance in \cite{zhang:1996} and \cite{owen+z:2000}. In addition, the present work focuses on \textit{self-normalized importance sampling} \citep[Chapter 9]{owen:13}, which only requires to know $f$ up to an unknown scale factor. This is particularly relevant for Bayesian estimation where the likelihood is known only up to a scale factor.  


The proposed approach, called \textit{safe adaptive importance sampling} (SAIS), follows from estimating the policy as a mixture between a kernel density estimate of $f$ (similar to \cite{west:1993}, \cite{givens+r:1996}, \cite{zhang:1996}, \cite{neddermeyer:2009}) and some ``safe" density with heavy tails compared to the ones of $f$. Such a modeling of the policy is motivated by the defensive importance sampling approach proposed in \cite{hesterberg:1995} (see also \cite{owen+z:2000}). Even though the kernel density estimate is flexible and adaptive, it remains risky as it depends on the location of the available particles: 
if these particles are away from the important places of $f$ 
then the resulting AIS estimate may result in a large variance. 
In contrast, the heavy-tailed density constitutes the safe part of the mixture as it is meant to alleviate the previous problem by allowing an exhaustive visit of the space during the algorithm. 
The tuning of the mixture parameter between these two densities will be a key ingredient in our work as this parameter will adjust the trade-off between variance efficiency, that is achieved when $q_{k}$ is close to $f$ (a result presented in the next section), and the exhaustiveness of the visit, that is achieved when $q_k$ has sufficiently large tails.

Another novelty of the paper is to consider the issue of policy learning as a functional approximation problem and our first theoretical contribution is the derivation of uniform convergence rates for $q_k$ estimating $f$.
Based on this, a central limit theorem is established for the resulting integral estimates. The success of the approach is illustrated by the asymptotic variance which is the same as the variance of the ``oracle'' procedure that would use the policy $q_k = f$ from the beginning. 
 
 In addition, a subsampling variant of SAIS, based on the \textit{sampling importance resampling} approach \citep{gordon+s+s:1993}, is introduced to decrease the computation time without deteriorating the performance of the method. We show that while having the same asymptotic behavior as SAIS, the subsampling variant requires only $n^{1+\delta}\log(n)$ operations ($\delta\in (0,1)$, defined in the next section, represents the degree of subsampling) whereas standard SAIS would need $n^2$ operations ($n$ is the number of requests to $f$).

From a theoretical stand point, the critical aspect of this work is to deal with the \textit{adaptive} character of the algorithms. 
The techniques used in the proofs bear resemblance with the ones developed in \cite{delyon+p:2018} where martingale tools have been used to study parametric AIS but more powerful results are required. Specifically, to handle kernel based estimates with importance weights, a modified version of Bennett's concentration inequality \citep{freedman:1975} turns out to be very useful.

The outline of the paper is as follows. In section \ref{sec:algo}, the mathematical framework is introduced, the algorithms are presented and illustrated. The main results are stated in section \ref{sec:main_results} and some comments are given in section \ref{sec:remarks}. Section \ref{sec:num_results} investigates their practical behavior. 
The mathematical proofs are gathered in Section \ref{sec:proof}.

\section{The SAIS framework}\label{sec:algo}

\subsection{Background}

Let $(X_k)_{k\geq 1}$ be a sequence of random variables defined on $(\Omega, \mathcal F,\mathbb P)$ and valued in $\mathbb R^d$.
The distribution of the sequence $(X_k)_{k\geq 1}$ is specified by its \textit{policy} as defined below. 

\begin{definition}
A policy is a sequence of probability density functions  $(q_k)_{k\geq 0}$ with respect to the Lebesgue measure adapted to the natural $\sigma$-field $(\mathcal F_{k})_{k\geq 0}$, $\mathcal F_{k} =\sigma(X_1,\ldots, X_{k})$ for $k\geq 1$, and $\mathcal F_0 = \emptyset$.  The sequence $(q_k)_{k\geq 0}$  is said to be the policy of  $(X_k)_{k\geq 1}$ whenever $X_k\sim q_{k-1}$, conditionally to $\mathcal F_{k-1} $.
\end{definition}

Let $f_U:\mathbb R^d \to \mathbb R_{\geq 0}$ be a measurable function such that $0<\int f_U < +\infty$ and define the associated probability density function $f = f_U / \int f_U$ (in this paper any integral is with respect to the Lebesgue measure). The function $f_U$ is called the \textit{unnormalized target} while the density $f$ is simply called the \textit{target}. The sequence $(W_k)_{k\geq 1}$ of \textit{importance weights} is defined by
\begin{align*}
&W_k = \frac{f_U (X_k) }{ q_{k-1}(X_k)}, \qquad \forall k\geq 1.
\end{align*}
Let $n\in \mathbb N^*$ denote the \textit{allocation}, i.e.,  the number of available requests to $f_U$. For any allocation $n\in \mathbb N^*$, the vector of \textit{normalized importance weights} $ (W_{n,k})_{k= 1,\ldots , n }$ is given by
\begin{align*}
\forall k=1,\ldots, n,\qquad &W_{n,k} \propto W_k  \quad  \text{such that}  \quad \sum_{k=1}^n W_{n,k} = 1.
\end{align*}
From the collection of weighted particles $(W_{k} , X_k)_{k=1,\ldots, n}$, integrals of the type $\int gf $ are computed as normalized estimate $\sum_{k=1} ^n W_{n,k} g(X_k)$.  The starting point of our approach is the following straightforward martingale property.
\begin{lemma}\label{prop:mg}
Let $g:\mathbb R^d \to \mathbb R$ be such that $\int |g|f <\infty$ and let $(q_k)_{k\geq 0}$ be the policy of $(X_k)_{k\geq 1}$. If, for any $k\geq 0$, $q_{k}$ dominates $f$, then the sequence 
\begin{align*}
 \left(\sum_{k=1}^n \left\{ W_k g(X_k)  - \int g f \right\} \right)_{n\geq 1} \quad \text{is a $(\mathcal F_n)_{n\geq 1}$-martingale} 
\end{align*}
  with quadratic variation $\sum_{k=1}^n V ( q_{k-1},g)$ where $V (q,g) = \int ( g^2 f ^2  / q )  - (\int g f ) ^2$.
\end{lemma}

\subsection{The auxiliary density estimation problem}

To choose the policy $q_k$ we consider an auxiliary problem: the estimation of $f$ with a kernel estimate. This is contrasting with other approaches that would focus on the estimation of $\int gf$ for a certain function $g$ and then choose $q_k$ accordingly (as proposed for instance in \cite{zhang:1996}). In that, we agree with one of the guidelines stated in \citep[section 1.2]{ho+b:1992} as the resulting approach will not depend on any integrand function $g$ but only on the target $f$. 

Let $K:\mathbb{R}^d \to \mathbb R_{\geq 0}$ be a density called \textit{kernel} and $(h_k)_{k\geq 1}$ be a sequence of positive numbers called \textit{bandwidths}. The kernel estimate of $f$ at step $n$, $f_n$, is given by
\begin{align*}
f_n (x) = \sum_{k=1}^n W_{n,k} K_{h_n} (x - X_k) ,\qquad x\in \mathbb R^d,
\end{align*}
where $K_h(u) = K(u/h)/h^d$. The estimate  $f_n$ is thus a mixture of $n$ densities centered on $X_k$ and having standard deviation $h_k$. Each weight $W_{n,k}$ reflects the importance of the associated particle $X_k$ within the mixture.

In the case where the policy is fixed, $q_k = q$ for all $k\in \mathbb N$,  weak convergence (denoted by ``$\leadsto $'') of the
previous estimate $f_n(x) $, for some $x\in \mathbb R^d$, can be easily obtained.  Define
\begin{align*}
 &\tilde f_n   =   (f* K_{h_n}) ,\\
& \sigma_q^2(x)  =  f(x)^2/ q(x) , \qquad x\in \mathbb R^d,
\end{align*}
where $*$ stands for the standard convolution product and with the convention $0/0 = 0$.  The proof of the following lemma is given in Section \ref{append1}.

\begin{lemma}[fixed policy]\label{lemma_fixed_policy}
Let $K:\mathbb R^d \to \mathbb R_{\geq 0}$ be a bounded probability density function and let $f$ and $q$ be continuous densities on $\mathbb R^d$ such that for some $c>0$,  $f(x) \leq c q(x)$ for all $x\in \mathbb R^d$.
Suppose that $q$ is the (constant) policy of $(X_k)_{k\geq 1}$. If $(h_k)_{k\geq 1} $ is a positive sequence such that $nh_n^d\to \infty$ and $h_n\to 0$, then for any $x\in \mathbb R^d$, it holds that, as $n\to \infty$,
\begin{align*}
(nh_n^d)^{1/2}
\left( f_n(x) -   \tilde f_n(x)    \right) \leadsto \mathcal N\left( 0 , \sigma_q^2(x) \int K^2 \right) .
\end{align*}
\end{lemma}

The choice of the policy $(q_k)_{k\geq 0}$, presented in the next section, will be guided by the following integrated-variance criterion 
\begin{align*}
C(q) = \int_{\mathbb R^d} \sigma_q^2(x) \, \diff x,
\end{align*}
which might be seen as an asymptotic version of the \textit{mean integrated squared error}, a popular criterion in density estimation \cite[section 3.1.2]{silverman:2018}.
Fortunately, its minimum is uniquely achieved when $q=f$ as stated in the following lemma, proof of which is given in Section \ref{append2}.

\begin{lemma}[variance optimality]\label{lemma:var_minimizer}
Let $f$ be a probability density function. The minimum of $C$ over the set of densities $q$ is achieved if and only if $q = f$ a.e.
\end{lemma}


\subsection{Standard SAIS}

The policy at use in standard SAIS is $(q_k)_{k\geq 0}$ defined for each $k\geq 1 $ as
\begin{align}\label{sampler_update}
q_{k}  = (1-\lambda_k)  f_k + \lambda_k q_0 
\end{align}  
where $(\lambda_k)_{k\geq 1}\subset [0,1]$ is a decreasing sequence of mixture weights and $q_0$ is the initial density. The component $q_0$ of the mixture allows to visit the space extensively during the algorithm. For this reason $q_0$ should be chosen with a sufficiently large tail compared to $f$. 
On the other side the value of $\lambda_k$ will decrease during the procedure in order to gain in efficiency, as Lemma \ref{lemma:var_minimizer} indicates. Balancing suitably between $f_k $ and $q_0$ enables to realize the trade-off, described in \cite{hesterberg:1995}, \cite{owen+z:2000}, between a tentatively optimal and a safe strategy.  Note in passing that generating from $K$ and $q_0$ allows to generate according to $q_k$. The algorithm is written below and an illustration is provided in Figure  \ref{figu:illust}.

\medskip
\noindent \textbf{Standard SAIS}\\
\textbf{Inputs}: The bandwidths $(h_k)_{k = 1,\ldots, n-1}$, the mixture weights $(\lambda_k)_{k = 1,\ldots, n-1}$, the initial density $q_0$
\smallskip\hrule\smallskip
 \noindent For $k= 0,1,\ldots, n-1 $:
 
 \smallskip
 \noindent generate $X_{k+1}$ from $q_k$ in \eqref{sampler_update} and compute $W_{k+1} = f_U(X_{k+1} ) / q_k(X_{k+1})$\\
\hrule
\medskip

Because computing each $q_{k}(X_{k+1})$ requires $O(k)$ operations, the cost of the previous SAIS algorithm is $O(n^2)$ operations plus $O(n)$ evaluations of $f_U$. When $f_U$ is hard to compute (e.g., Bayesian likelihood), this last contribution may dominate. In contrast, when a request to $f_U$ represents a single operation, the $n^2$ operations could be prohibitive compared to other approaches such as parametric AIS. To widen the applicability of SAIS, we propose a subsampling version whose main purpose is to decrease the computational cost of the initial version without reducing the method performances.

\subsection{Subsampling SAIS}

To decrease the number of particles, we follow the \textit{sampling importance resampling} approach proposed in \cite{gordon+s+s:1993}. At each step $n$,  a bootstrap sample $(X_{n,k}^*)_{k= 1,\ldots,  \ell_{n}}$ of (small) size $ \ell_n$ (compared to $n$) with distribution $  \mathbb P_n = \sum_{k=1} ^{n} W_{n,k} \delta_{X_k} $ is drawn. The associated kernel estimate of $f$ is then defined as
\begin{align*}
f_{n}^*(x) = \ell_{n}^{-1} \sum_{k=1}^{\ell_{n}} K_{h_n} (x-X_{n,k}^*) ,\qquad x\in \mathbb R^d,
\end{align*}
and the policy at use in the subsampling version is simply given by, for $k\geq 1$,
\begin{align}\label{sampler_update_boot}
q_{k}  = (1-\lambda_k)  f_k^* + \lambda_k q_0.
\end{align}
The procedure is fully described in the algorithm below.
 

\medskip
\noindent \textbf{Subsampling SAIS}\\
\textbf{Inputs}: The bandwidths $(h_k)_{k = 1,\ldots, n-1}$, the mixture weights $(\lambda_k)_{k = 1,\ldots, n-1}$, the size of the bootstrap sample $(\ell_k)_{k = 1,\ldots, n-1}$, the initial density $q_0$
\smallskip\hrule\smallskip
\noindent Initialization: generate $X_1$ from  $q_0$\\
 \noindent For $k= 1,2,\ldots, n-1 $:
 \hspace{3cm}\begin{itemize}\itemsep10pt
\item  generate an independent and identically distributed sample $(X_{k,i}^*)_{i = 1,\ldots,  \ell_{k}}$ from $  \mathbb P_k = \sum_{i=1} ^{k} W_{k,i} \delta_{X_i}  $
\item generate $X_{k+1}$ from $q_k$ in \eqref{sampler_update_boot} and compute $W_{k+1} = f_U( X_{k+1} ) / q_{k}( X_{k+1} ) $
\end{itemize}

\hrule
\medskip

At each iteration $k$ of the previous algorithm, we can keep in memory the cumulative sum $\sum_{i=1} ^k W_i$ which is updated by a single operation. A uniform variable over $[0, \sum_{i=1} ^ k W_i ]$ is drawn and localized within the cumulative sums. This costs us $\log(k)$ operations (with a dichotomic search) and gives us one multinomial draw. Since $\ell_k$ such draws must be done, we have $\ell_k\log(k)$ operations to obtain the sample $(X_{k,i}^*)_{i = 1,\ldots,  \ell_{k}}$. Then, we draw one point $X_{k+1}$ according to $q_k$ in (\ref{sampler_update_boot}) which is only one operation because $f^*_k$ is a mixture with equal weights. Evaluating $q_{k}(X_{k+1}) $ is $\ell_k$ more operations. Hence $O(\ell_k \log(k))$ operations are needed at each iteration $k$ implying that the total number of operations is of order
$\sum_{k=1}^n  \ell_k \log(k)$. In the simulation study we shall work with $\ell_n = n^{\delta}$, $\delta\in (0,1)$, leading to a computation cost of order $ n^{1+\delta} \log (n )$.

\begin{figure}
\centering\includegraphics[height=3.5cm,width = 3.5cm]{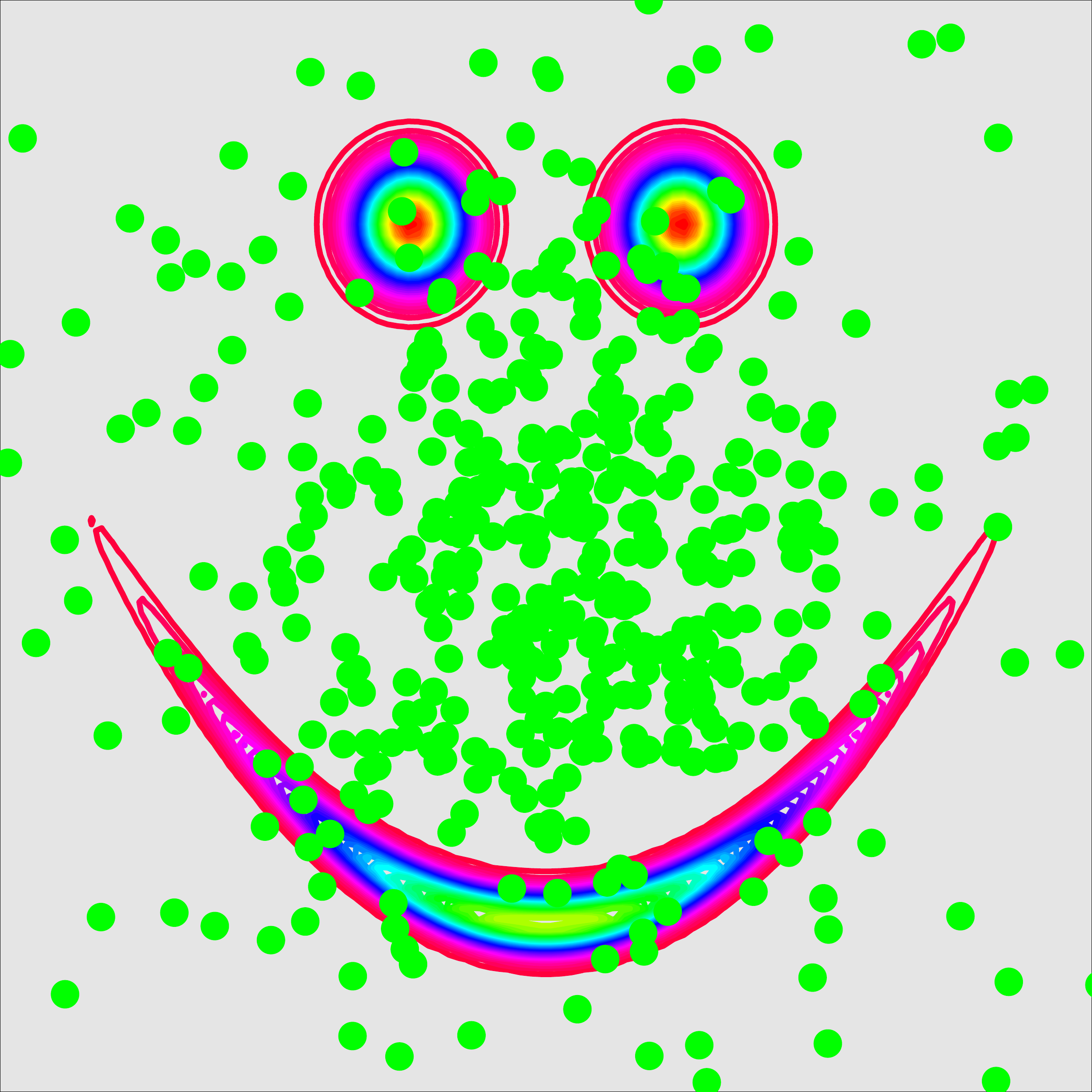}\centering\includegraphics[height=3.5cm,width = 3.5cm]{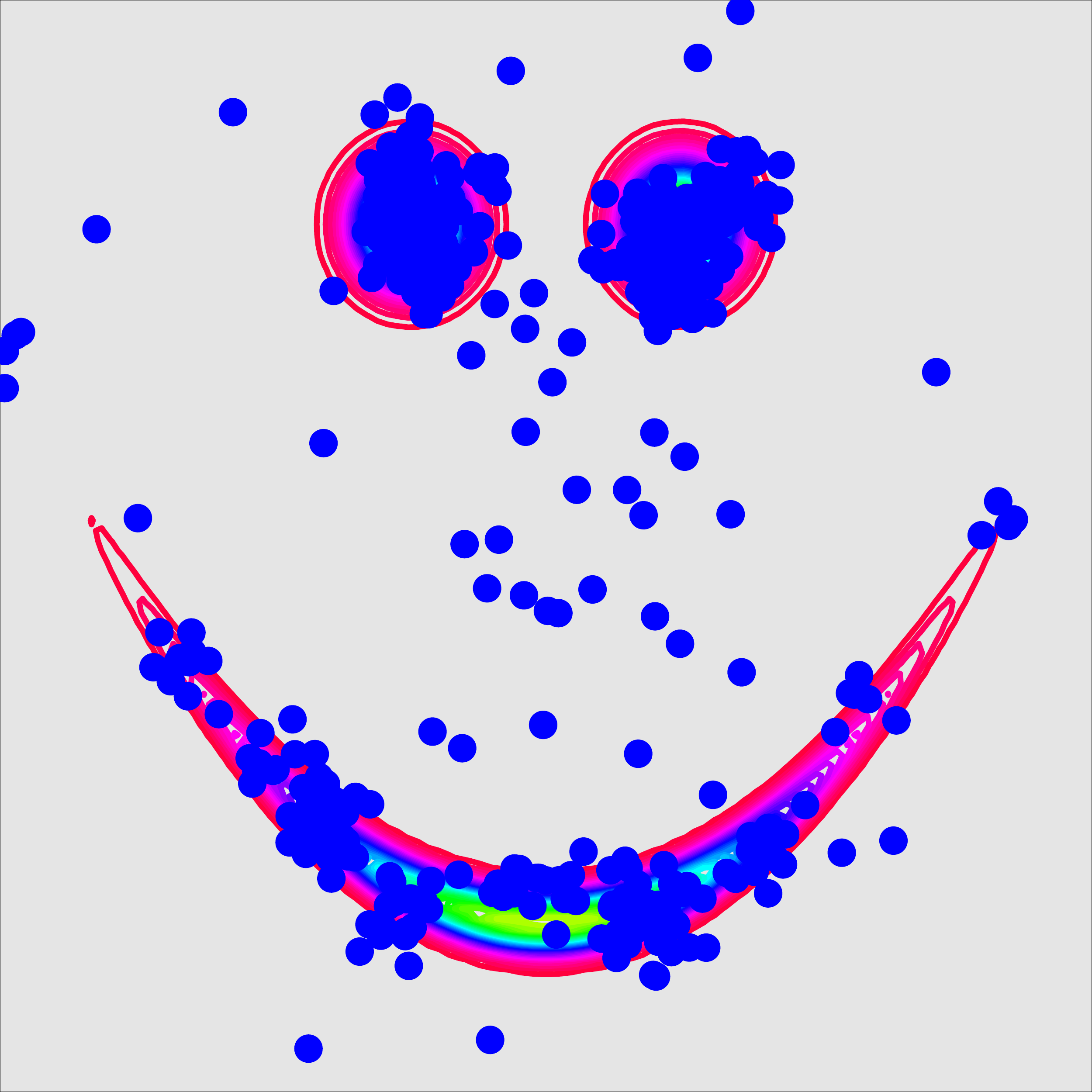}\centering\includegraphics[height=3.5cm,width = 3.5cm]{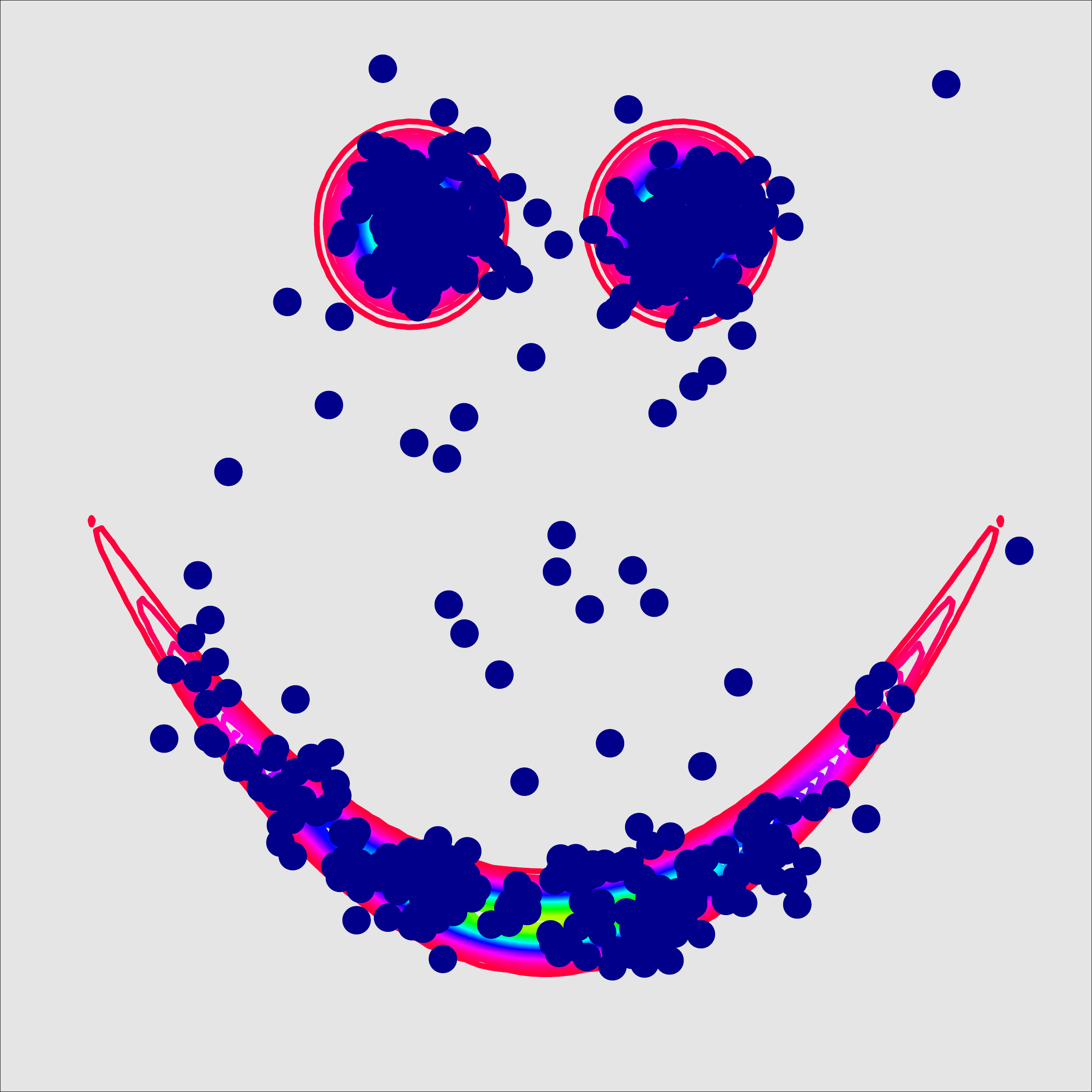}
\centering\includegraphics[height=2.9cm,width = 3.5cm]{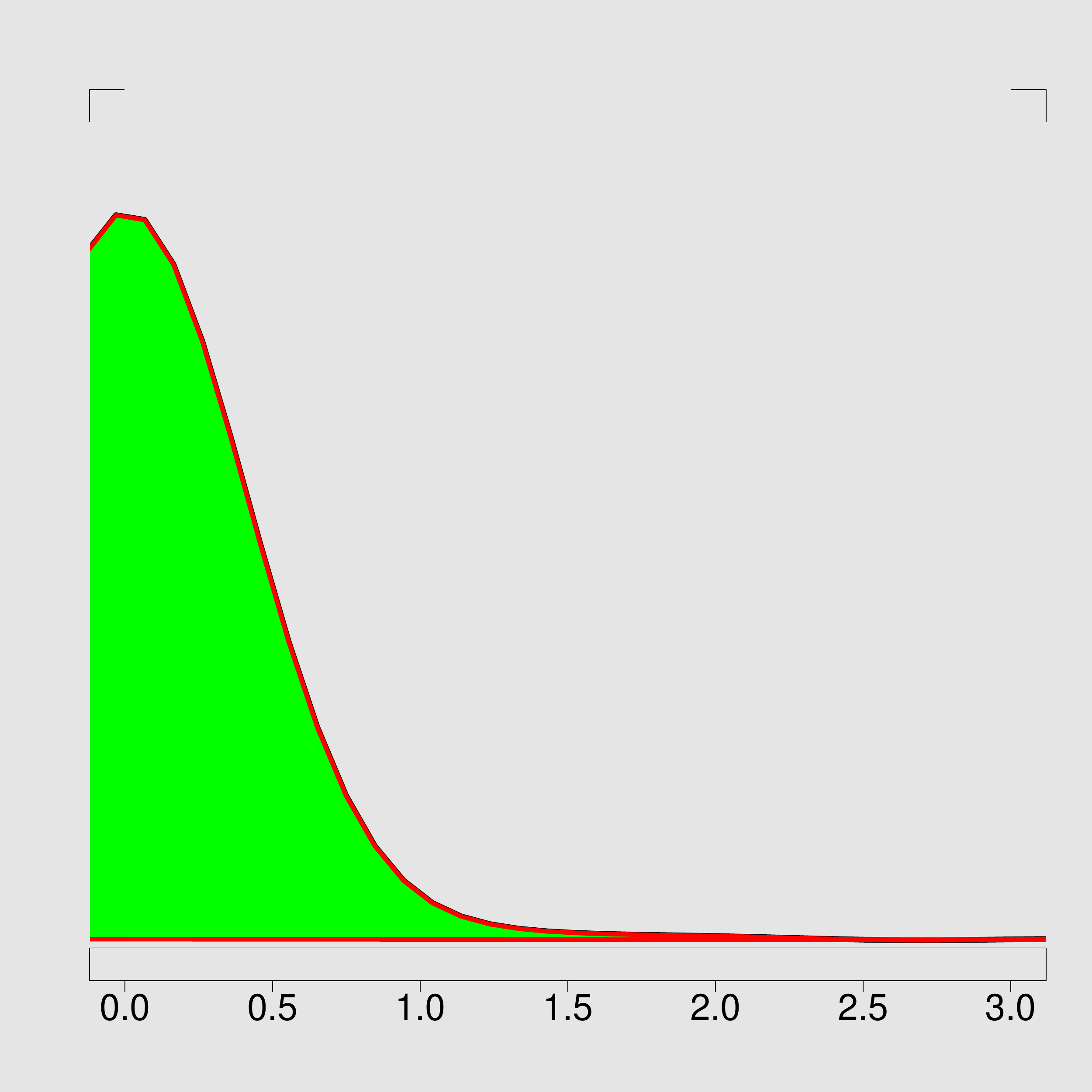}\centering\includegraphics[height=2.9cm,width = 3.5cm]{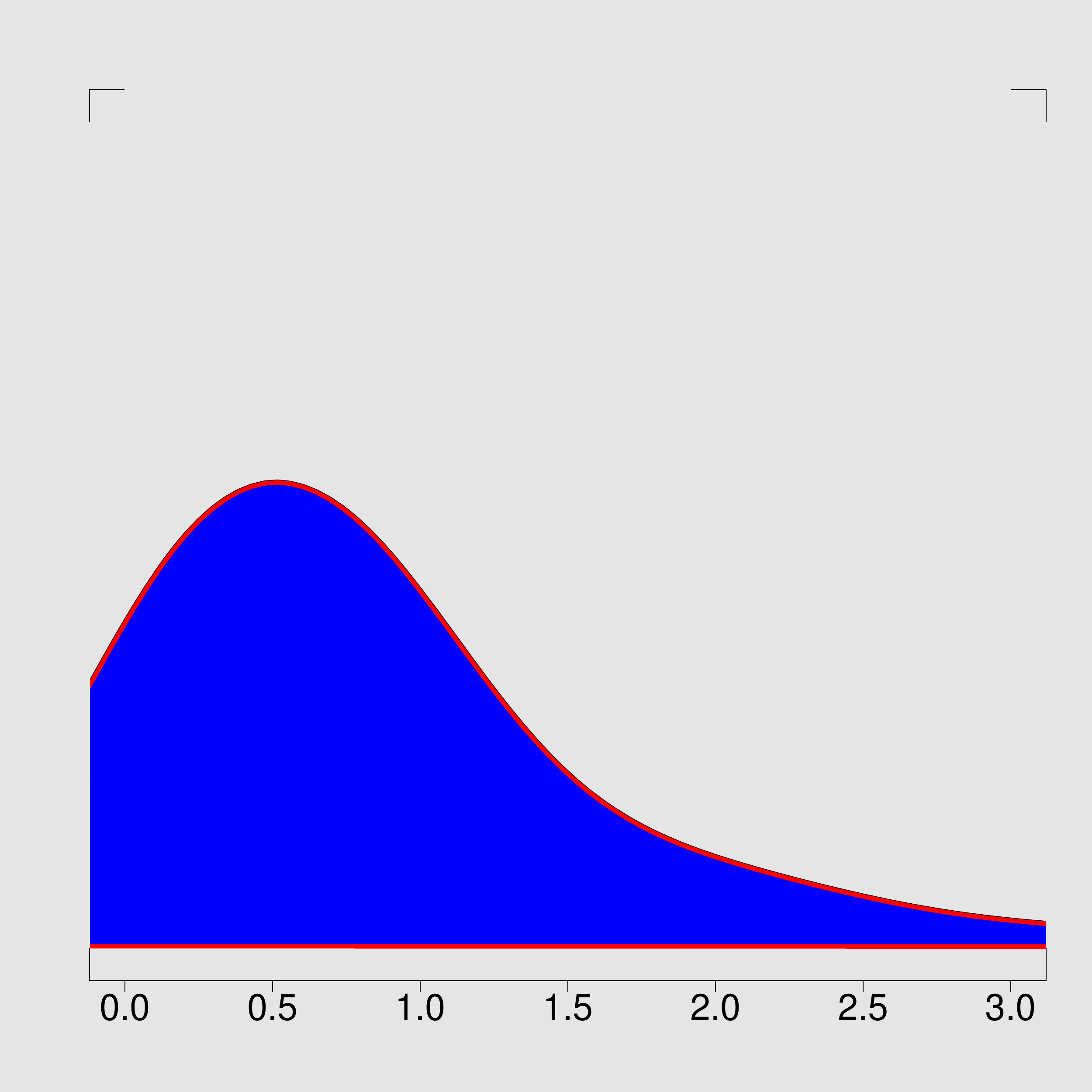}\centering\includegraphics[height=2.9cm,width = 3.5cm]{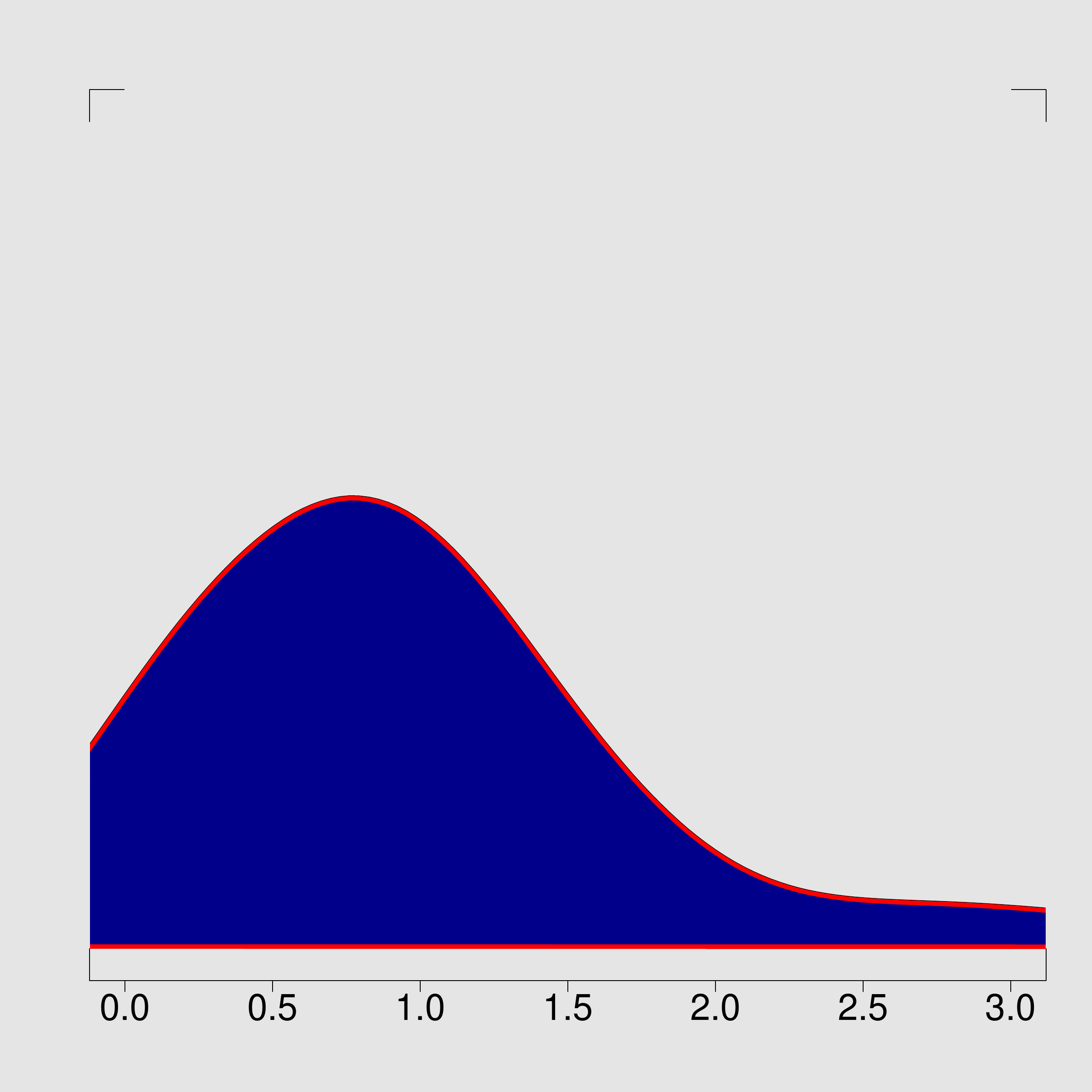}
\caption{\label{figu:illust}
The target $f$ is a mixture of a banana shaped distribution and two Gaussians, The $\lambda_k$ are equal to $1$ for $k=0,\ldots, 499$, to  $0.1$ for $k=500,\ldots, 999$ and to  $0.05$ for $k =1000,\ldots, 1499$. The $h_k$ are equal to $1$ all along the algorithm. The initial density $q_0$ is a standardized Student's $t$-distribution. The particles indexed by $k=1,\ldots, 500$ are on the left, $k=501,\ldots, 1000$ are in the middle, $k =1001,\ldots, 1500$ are on the right. At the bottom are density plots of the corresponding importance weights $W_k$.}
\end{figure}

\section{Asymptotics of SAIS}\label{sec:main_results}

\subsection{Assumptions}\label{sec:martingale}

The sequences $(\lambda_k)_{k\geq 1}$, $(h_k)_{k\geq 1}$ and $(\ell_k)_{k\geq 1}$ are such that for any $k\geq 1$, $h_k$ is bounded and positive, $\lambda_k\in (0,1]$ is nonincreasing and $1\leq \ell_k\leq k$ is nondecreasing.
For clarity reasons, the additional assumptions related to the sequences $(h_k)_{k\geq 1}$, $(\lambda_k)_{k\geq 1}$ and $(\ell_k)_{k\geq 1}$ will be stated within the statements of each result.
The results of the paper are expressed using the sequence
\begin{align*}
a_n =  \sqrt{\frac{\log(n)}{nh_n^d}}, \qquad n \geq 1, 
\end{align*}
and the following standard notation: for two nonnegative sequences $(u_n)_{n\geq 1}$ and $(v_n)_{n\geq 1}$, $u_n\ll v_n$ or $u_n = o(v_n)$ means $u_n  / v_n \to 0$ as $n\to \infty$; $u_n = O(v_n)$ means that $u_n / v_n $ is bounded. The Euclidean norm is denoted by $\|\cdot\|$. The assumptions on $f_U$, $q_0$ and $K$ are given below. For clarity the assumptions are stated with respect to $f = f_U / \int f_U $ rather than $f_U$. They would have been the same using $f_U$.

\begin{enumerate}[label=(\text{H}\arabic*),resume=count_cond]
\item \label{cond:f}  $f$ is a probability density function on $\mathbb R^d$ two times continuously differentiable, with bounded second derivatives.

\item \label{cond:f_tail}   $ q_0 $ is a probability density function on $\mathbb R^d$. For any $\eta\in (0, 1]$, there exists $c_\eta>0 $ such that for all  $x\in\mathbb R^d$ 
\begin{align*}
 f(x)^\eta  \leq c_\eta q_0(x).
\end{align*}
In addition, there exist $C_0$ and $r_0$ positive numbers such that for all  $x\in\mathbb R^d$ 
\begin{align*}
q_0(x) \leq C_0  (1+\|x\|)^{-r_{0}} .
\end{align*}

\item\label{cond:K} $K:\mathbb R^d\to \mathbb R_{\geq 0}$ is a Lipschitz probability density 
function such that
\begin{align*}
& \int uK(u) \,\diff u = 0,\qquad \int \|u\|^2K(u)\,\diff u<\infty.
\end{align*}
In addition, there exist $C_K$ and $r_K$ positive numbers such that for all  $x\in\mathbb R^d$
\begin{align*}
&K(x)\le C_K(1+\|x\|)^{-r_K}.
\end{align*}
\end{enumerate}

\subsection{Preliminary results with general policy}

Let us now present some results that are valid for a general policy 
\begin{align}\label{eq:general_policy}
q_k = (1- \lambda_k) f_k^\circ + \lambda_k q_0,\qquad k\geq 1,
\end{align}
where $(f_k^\circ)_{k\geq 1}$ is any sequence of densities adapted to the filtration $(\mathcal F_k)_{k\geq 1}$. They will be useful in the analysis of standard SAIS (equation \eqref{sampler_update}) as well as in the study of the subsampling version (equation \eqref{sampler_update_boot}). The proof of the following lemma is presented in Section \ref{append:initial}.

\begin{lemma}[initial bound]\label{lemma:initial_consistency}
Assume \ref{cond:f}, \ref{cond:f_tail}, \ref{cond:K} and work under a general policy \eqref{eq:general_policy}. If $a_n^2 \ll \lambda_n$, then we have,
\begin{align*}
&\sup_{x\in\mathbb R^d}|f_{n}(x)-f(x)|=O\big(a_n\lambda_n^{-1/2}+h_n^2\big)\qquad \text{a.s.}
\end{align*}
\end{lemma}

The next result, whose proof is given in Section \ref{append:th:1}, 
provides a sufficient rate of proximity between $f^\circ_n$ and $f$ allowing to improve the previous initial bound. 
That is, we shall assume that there exists $\varepsilon\in (0,1/2]$ such that
\begin{align}\label{qfhyp}
&\sup_{x\in\mathbb R^d}|f^\circ_{n}(x)-f(x)|=O\big(\lambda_n^{1/2+\varepsilon}\big)\qquad \text{a.s.}
\end{align}

\begin{lemma}[improved bound]\label{lemma:1}
Assume \ref{cond:f}, \ref{cond:f_tail}, \ref{cond:K} and work under a general policy \eqref{eq:general_policy} that satisfies \eqref{qfhyp}.  If $a_n\ll \lambda_n$, we have
\begin{align*}
&\sup_{x \in \mathbb R^{d}}|f_n(x)-f(x)|=O(a_n+h_n^2), \qquad \text{a.s.}
\end{align*}
\end{lemma}

We terminate the section by two central limit theorems. Both require \eqref{qfhyp} to hold in order to obtain the appropriate asymptotic variance, i.e., the one associated to the policy $f$. The proof is in Section \ref{append:lemma:2}.

\begin{lemma}\label{lemma:2}
Assume \ref{cond:f}, \ref{cond:f_tail}, \ref{cond:K} and work under a general policy \eqref{eq:general_policy} that satisfies \eqref{qfhyp}.  If $ a_n^2 \ll \lambda_n \ll 1$,  then for any function $g:\mathbb R^d \to \mathbb R$ such that $\int g^2q_0<\infty$, as $n\to \infty$,
\begin{align}\label{gnorm}
&\sqrt {n}\left(\sum_{k=1} ^nW_{n,k}g(X_{k})-\int g f\right)
  \leadsto\mathcal N\big(0,V(f,g)\big)  ,
\end{align}
where $V(f,g)$ is defined in Lemma \ref{prop:mg}.
\end{lemma}

\subsection{Main results}

Now we focus on the SAIS algorithms presented in Section \ref{sec:algo}. 
We start by considering standard SAIS. The subsampling version will be studied right after. 

Assuming that for some $\epsilon>0$, $ a_n +h_n^4\ll\lambda_n^{1+\varepsilon}$, we can apply Lemma \ref{lemma:initial_consistency} to obtain that (\ref{qfhyp}) is valid with $f_n$ in place of $f_n^\circ$. This permits to apply Lemma \ref{lemma:1} 
and leads to the following two results.

\begin{theorem}[uniform convergence rate]\label{th:1}
Assume \ref{cond:f}, \ref{cond:f_tail}, \ref{cond:K} and work under policy \eqref{sampler_update}.  If there exists $\epsilon>0$ such that $a_n +  h_n ^4 \ll \lambda_n^{1+\epsilon}\ll 1   $, we have with probability $1$
\begin{align*}
&\sup_{x \in \mathbb R^{d}}|f_n(x)-f(x)|=O(a_n+h_n^2).
\end{align*}
\end{theorem}

We now consider the weak convergence of the sequence $\sum_{k=1}^n W_{n,k} g(X_k)$ for which we derive the asymptotic variance $V(f,g)$. The proof is a simple application of Lemma~\ref{lemma:initial_consistency} and Lemma~\ref{lemma:2}, equation \eqref{gnorm}.

\begin{theorem}[asymptotic normality of integral estimates]\label{th:cvd2}
Assume \ref{cond:f}, \ref{cond:f_tail}, \ref{cond:K} and work under policy \eqref{sampler_update}. If there exists $\epsilon >0$ such that  $a_n +  h_n ^4 \ll \lambda_n^{1+\epsilon}   \ll 1$ and if $ \int g^2 q_0 <\infty$, we have as $n\to \infty$, then \eqref{gnorm} is valid.
\end{theorem}

Now let us consider SAIS with subsampling. The following result, whose proof is in Section \ref{app:proof_cvd3}, shows that subsampling is efficient as the convergence rate and the asymptotic variance are the same in Theorem \ref{th:cvd2}.

\begin{theorem}[asymptotic normality of integral estimates with subsampling]\label{th:cvd3}
Assume \ref{cond:f}, \ref{cond:f_tail}, \ref{cond:K} and work under the subbsampling policy \eqref{sampler_update_boot}.
If there exists $\epsilon >0$ such that  $a_n^2n/\ell_n + a_n + h_n ^4 \ll \lambda_n^{1+\epsilon}\ll 1$ and if $ \int g^2 q_0 <\infty$, then \eqref{gnorm} is valid.
\end{theorem}

\section{From theory to practice}\label{sec:remarks}

In this section, some comments dealing with the main results are given. Then we consider the practical tuning of the SAIS algorithms.

\subsection{General comments on the main results}

\paragraph{Related literature on kernel smoothing}
Theorem \ref{th:1} is related to uniform convergence results for kernel density estimates for independent sequences \citep{gigui2001,gine+g:02}, weak dependent sequences \citep{hansen:2008}, and Markov chains \citep{azais+d+p:2018,bertail+p:2018}. In these papers, the same rate of convergence, $\sqrt{\log(n) / (nh_n^d)}$ (for the variance term), was obtained but under different assumptions on the dependence structure of the considered sequences.

\paragraph{Related literature on the asymptotics of AIS}
The asymptotic regime that has been considered allows the distribution $q_k$ to change at each stage of the algorithm. This is similar to the regimes studied in \cite{ho+b:1992}, \cite{delyon+p:2018}, \cite{feng+m+s+w:2018} but different from the results presented in \cite{chopin:2004}, \cite{douc+g+m+r:2007a}, \cite{douc+g+m+r:2007b} where the sequence $(q_{k})_{k=0,\ldots, n-1}$ is frozen after a certain given time (i.e., the number of updates is finite), or from \cite{marin+p+s:2019}  where a variant of AIS is studied when the number of samples between each update is an increasing sequence going to $\infty$. Finally in \cite{zhang:1996}, the author works under another asymptotic regime in which the number of particles generated in the first stage has the same order as the total amount of particles.


\paragraph{Curse of dimensionality}
Theorems \ref{th:1} and \ref{th:cvd2} are of a different nature. Theorem \ref{th:1} is dealing with functional estimation and, consequently, is subjected to the well-known curse of dimensionality \citep{stone:1980}. 
In contrast, the weak convergence result stated in Theorem \ref{th:cvd2}, 
which is concerned with the estimation of a single parameter, $\int gf$, is not impacted by the value of $d$. 
This is because the estimation error between $f_n$ and $f$ intervenes at a second order in the decomposition used in the proof. 
This very last point motivated the subsampling version as it supports the use of \textit{rough but cheap} strategies 
for the estimation of $f$.

\paragraph{Choice of the kernel}
The kernel $K$ being non-negative (this is needed to ensure easy random generation according to $q_n$), 
it cannot have more than one vanishing moment. This bounds the exploitable smoothness of $f$ to two derivatives and explains 
why the rate of decrease of the bias term is ${h_n^2}$. 

\paragraph{Asymptotic normality of the kernel estimate}
Concerning the kernel estimate $f_n$, only uniform convergence results have been presented so far but the choice of the policy \eqref{sampler_update} has been motivated initially by a weak convergence result (see Lemma \ref{lemma_fixed_policy} and \ref{lemma:var_minimizer}). A question that remains is to know whether the optimal variance given in Lemma \ref{lemma:var_minimizer}, i.e., $f(x) \int K^2$ is achieved when the policy \eqref{sampler_update} is put to work. The answer is positive as stated in Theorem \ref{th:cvd2} given in Appendix \ref{sec:weakcvfn}.

\paragraph{The compact case}
When $f$ is compactly supported and bounded away from $0$, 
the study of the algorithm is simpler and similar results are valid under weaker conditions 
on $\lambda_n$ and $h_n$. This is presented in Appendix \ref{sec:compact_case}.

\subsection{Practical details}

\paragraph{Choice of $(\lambda_n,h_n)$ for standard SAIS} 
If $h_n\propto n^{-\kappa}$ for some $\kappa < 1$, then, 
balancing the variance term $a_n$ (up to a log) and the bias term ${h_n^2}$
in Theorem~\ref{th:1} leads to $$h_n\propto n^{- 1/(4+d) }. $$ The corresponding rate, $ n^{-2/(4+d)}$ (up to a log),
is the usual optimal rate in non parametric estimation when the function is at least $2$-times continuously differentiable and the kernel has order $2$  \citep{stone:1980}.



In practice, a slow decrease of $\lambda_n$ would favor an exhaustive visit of the space during what could be called the burn-in phase of the algorithm. Such a tuning of $\lambda_n$ might be appropriate when facing a difficult problem e.g., several modes or large variance of $f$. In contrast, a rapid decrease of $\lambda_n$ could be risky because the algorithm is likely to miss some important parts of the distribution.
In Theorem \ref{th:cvd2}, allowing $\lambda_n$ to go to $0$ only influences the asymptotic variance but the convergence rate remains the same. For instance, if $(\lambda_k)_{k\geq 1}$ was converging to a constant $\lambda_0 > 0$, one would get $\sigma_q^2(x) \int K^2 $, with $q = (1-\lambda_0)f+\lambda_0 q_0$, as asymptotic variance in Theorem~\ref{th:cvd2} which would be fine in many cases as soon as $\lambda_0$ is small enough.

As expressed in  Theorem \ref{th:cvd2}, the only restriction we have on $\lambda_n$ is that it goes to $0$ not too quickly, i.e., $a_n + h_n^4\ll \lambda_n^{1+\epsilon}\ll 1$, for some $\epsilon >0$. 
Under the optimal bandwidth $h_n \propto   n^{-1/(4+d)}$, an appropriate choice is $$\lambda_n \propto  n^{-1/(4+d)}.$$

\paragraph{Choice of $(\lambda_n,h_n,\ell_n)$ for SAIS with subsampling} 
We discuss the subsampling version with $\ell_n = n^{\delta}$ (up to some rounding), where $\delta\in (0,1/2]$ encodes for the degree of subsampling. From Theorem \ref{th:cvd3}, it is reasonable to balance between variance and bias, $a_n\sqrt{n/\ell_n}$ and $h_n^2$, to choose the bandwidth 
$$h_n \propto \ell_n^{-1/(4+d)}. $$ 
Reasonably, the size of the bandwidth increases with the degree of subsampling. When $\delta\leq 1/2 $, $a_n  \propto n^{- (4 +d(1-\delta)) /(8+2d)} $ (up to a log) is negligible before $a_n\sqrt{n/\ell_n} \propto n^{-4\delta/(4+d)}$. Following Theorem \ref{th:cvd3}, one can set
$$\lambda_n  \propto \ell_n^{-2/(4+d)}.$$ 
Note that when $\delta = 1/2$, we recover the same value as the one recommended for standard SAIS.

\paragraph{Updating $q_0$}
A variant of the proposed approach is to use $q_0(\cdot - \mu_k )$ with $\mu_k = \sum_{i=1} ^ k  W_{k,i} X_i$ instead of $q_0$ in the policy \eqref{sampler_update}.
Such a policy should increase efficiency while staying in an heavy-tailed family of densities. This can be handled by a slight modification of our proofs. Because the sampler $q_0(\cdot - \mu_k )$ now depends on $\mathcal F_k$, the condition needed, \ref{cond:f_tail}, must be satisfied uniformly over $k$. This is certainly the case whenever $\mu_k$ is restricted to a compact set of $\mathbb R^d$.

%

\paragraph{Mini-batching}
This is a common extra ingredient of AIS schemes. It consists in grouping the particles into ``mini-batches'' in which the particles have the same distribution. In other words, the  policy $q_{k}$ is frozen over these mini-batches and the update  
of $q_{k}$ is conducted only when $k$ is entering a new mini-batch. This will save the time needed to update and allow to run in parallel the generation of the random variables according to $q_k$.
For clarity reasons, the theoretical study to SAIS has been restricted to the case when each mini-batch is made of one sample point. 
The extension to mini-batches of size $m\geq 1$ with $mT = n$ can be carried out easily. Suppose that $(h_k)_{k\geq 1}$ and $(\lambda_k)_{k\geq 1}$ are just as in Lemma \ref{lemma:initial_consistency}. Define the sequence $(\tilde q_t)_{t\geq 1} $ such that  $\tilde q_t = (1-\tilde \lambda_t) \tilde f_t + \tilde \lambda_t q_0$ where for each $t=1,\ldots, T$, $\tilde \lambda_t = \lambda_{mt }$, $\tilde h_t = h_{mt }$ and $\tilde f_t = f_{mt}$. The mini-batch algorithm corresponds to the standard SAIS algorithm using 
$(q_k)_{k\geq 1}$ given by
$$\tilde q_0,\ldots, \tilde q_0, \tilde q_1,\ldots, \tilde q_1, \ldots $$
Hence we can apply  Lemma \ref{lemma:initial_consistency}. With the obtained convergence rate, we can proceed similarly as before: apply Lemma \ref{lemma:1} and \ref{lemma:2}, just as it has been done for proving Theorem \ref{th:1} 	and \ref{th:cvd2}.


\section{Simulation study}\label{sec:num_results}

The aim of the section is to illustrate the practical behavior of the SAIS algorithms. For the sake of reproducibility, we start by describing precisely the algorithms at use.
Then two basic examples will be considered. 

\subsection{Algorithms}\label{sec:variants}


Here we bring together the pieces of information gathered in Section \ref{sec:remarks} to write down our ultimate SAIS algorithms (with and without subsampling). These very algorithms will be at use in the simulation study.
The allocation, $n\in \mathbb N^*$, is made of $T\in \mathbb N^*$ mini-batches containing $m_0= n/T \in \mathbb N^*$ particles. The set of particles indexes of any stage $t= 1,\ldots, T $ is
$$ B_t =\{ 1+ (t-1) m,\ldots, tm\}.$$ 
Let $(h_t)_{t = 1 ,\ldots, T-1}$ be the sequence of bandwidths and $(\lambda_t)_{t = 1,\ldots, T-1}$ be the sequence of mixture weights. For any $t = 1,\ldots, T$, define the discrete distribution associated to the weighted particles as
\begin{align*}
&\mathbb P_t \propto \sum_{k=1}^{mt}W_{k} \delta_{X_k}.
\end{align*}
The corresponding mean value is denoted $\mu_t$ and the associated kernel density estimate is given by
\begin{align*}
f_t(x)  = \int K_{h_t} (x-y)  \mathbb P_t(dy) , \qquad x\in \mathbb R^d . 
\end{align*}
The policy at use in the mini-batch version is given by, for $t\geq 1$,
\begin{align}\label{eq:mini_batch_policy}
 q_{t} (x) =  (1-\lambda_{t} ) f_{t}(x) + \lambda _{t} q_0( x - \mu_{t})  ,\qquad x\in \mathbb R^d . 
\end{align}
The algorithm includes a burn-in phase which corresponds to the first $T_0$ stages and, roughly speaking, aims at giving a tour in the target's domain. 
During this early phase, the number of points is very small and the importance weights $W_k$ might have a large variance. To avoid the (uncommon) situation where a few weights carry all the mass of $\mathbb P_t $, we use \textit{regularized weights} $W_k^{\eta}$, with $\eta\in (0,1)$, instead of $W_k$. This simple operation will uniformize the weights. The algorithm writes as follows.

\medskip
\noindent \textbf{SAIS with mini-batching}\\
\textbf{Inputs}: Allocation $n$, number of stages $T$, bandwidths $(h_t)_{t = 1,\ldots, T-1}$, mixture weights $(\lambda_t)_{t = 1,\ldots, T-1}$, density $q_0$,  initial mean $ \mu_{\text{start}}$, burn-in parameters $(T_0,\eta)$.
\smallskip\hrule\smallskip
\noindent Initialize $\mu_0 = \mu_{\text{start}}$. For $t= 0,1,\ldots, T-1 $:
\hspace{3cm}\begin{itemize}\itemsep10pt
\item generate $(X_{k})_{k\in B_{t+1} }$ from $ q_{t} $ in \eqref{eq:mini_batch_policy}; compute $ W_{k} =  f_U(X_k) /  q_{t} (X_k) $;
\item if $t\leq T_0$, set $W_{k} = W_{k} ^{\eta}$ for all $k\in B_{t+1}$.
\end{itemize}
\hrule
\medskip

In SAIS with subsampling, the policy is given by, for $t\geq 1$,
\begin{align}\label{eq:mini_batch_policy_sub}
 q_{t} (x) =  (1-\lambda_{t} ) f_{t}^*(x) + \lambda _{t} q_0( x - \mu_{t})  ,\qquad x\in \mathbb R^d ,
\end{align}
where, in contrast with the standard SAIS, a bootstrap step will be needed at each stage to provide $f_t^*$, a bootstrap kernel estimate of $f$. The algorithm is written below.

\medskip
\noindent \textbf{SAIS with mini-batching and subsampling}\\
\textbf{Inputs}: Allocation $n$, number of stages $T$,  subsampling size $(\ell_t)_{t = 1,\ldots, T-1}$, bandwidths $(h_t)_{t = 1,\ldots, T-1}$, mixture weights $(\lambda_t)_{t = 1,\ldots, T-1}$, density $q_0$, initial mean $ \mu_{\text{start}}$, burn-in parameters $(T_0,\eta)$.
\smallskip\hrule\smallskip
\noindent Initialize $\mu_0 = \mu_{\text{start}}$. For $t= 0,1,\ldots, T-1 $:
\hspace{3cm}\begin{itemize}\itemsep10pt
\item if $t\geq 1$, generate $ (X_k^* ) _ {k = 1,\ldots, \ell_t} $ from $\mathbb P_{t}$; set $\mathbb P_{t} ^* \propto \sum_{k=1} ^{\ell_t} \delta_{X_k^*} $ and $$ f_t^*(x)  =    \int K_{h_t} (x - y) \mathbb P_{t}^*(dy);$$
\item generate $(X_{k})_{k\in B_{t+1} }$ from $ q_{t} $ in \eqref{eq:mini_batch_policy_sub}; compute $ W_{k} =  f_U(X_k) /  q_{t} (X_k) $;
\item if $t\leq T_0$, set $W_{k} = W_{k} ^{\eta}$ for all $k\in B_{t+1}$.
\end{itemize}
\hrule
\medskip

\subsection{Methods in competition}

In the simulation study, each method will be compared with an overall allocation equal to $ n = 200000$.

\paragraph{Standard SAIS} 

There is $T = 200$ stages each generating $m = 1000$ particles. For the burn-in phase, we take $T_0 = 20$ and $\eta = 3/4$. Define $n_0 = 10000$. When no subsampling is performed, in agreement with Section \ref{sec:remarks}, the values of $(h_t,\lambda_t)$ are given by 
\begin{align*}
&h_t = (0.4/\sqrt d )   ( 1 + mt / n_0 ) ^{-1/(4+d)}\\
&\lambda_t = 0.25  ( 1 + mt / n_0 ) ^{-1/(4+d)}
\end{align*}
for all $t= 1,\ldots, T-1$ except for $\lambda_t$ during the burn-in phase where we simply set $\lambda_t = 1$, $t=1,\ldots, 9$ and $\lambda_t = .5$, $t=10,\ldots, 19$. In the definition of the bandwidth, the factor $.4/\sqrt d $ corresponds to the standard error estimate in the \textit{Silverman's rule of thumb} \cite{silverman:2018}. Along the section this method is denoted \texttt{SAIS}.

\paragraph{Subsampling SAIS}  Following  Section \ref{sec:remarks} recommendation, at each iteration $t= 1,\ldots, T-1$, the subsampling is carried out with $\ell_t =  10  \lfloor   { (n_0 + mt)^\delta} \rfloor$, where $\delta\in (0,1/2]$. Two versions of subsampling SAIS will be considered: \texttt{SAIS*2} is when  $\delta = 1/2$ and \texttt{SAIS*4} is when  $\delta = 1/4$. The values of $(h_t,\lambda_t)$ are given by 
\begin{align*}
&h_t = (0.4/\sqrt d )   ( 1 + \ell_t / n_0 ) ^{-1/(4+d)}\\
&\lambda_t = 0.25  ( 1 +\ell_t / n_0) ^{-2/(4+d)}
\end{align*}
for all $t= 1,\ldots, T-1$ except during the burn-in phase, where $\lambda_t$ is tuned as before. 

In both SAIS, $q_0$ is the student distribution with covariance matrix $(5  / d) I_d $. The initial mean value $\mu_{\text{start}}$ will change depending on the considered example.


\paragraph{Metropolis-Hastings} Denote by $\phi_{\Sigma}$ the Gaussian density with mean $0$ and covariance matrix $\Sigma$.
A natural competitor is the Metropolis-Hastings algorithm for which two proposals are considered. The random walk version is when the proposal is $\phi_{\Sigma} $ where $\Sigma = (0.4  / d) I_d $. The adaptive Metropolis-Hastings, as introduced in \cite{haario+s+t:2001}, denoted \texttt{AMH}, is when the proposal is $\phi_{\Sigma_n} $ with $\Sigma_n $ the estimated covariance matrix based on the past iterations of the chain. The adaptive proposal is put to work from iteration $1e4$. Before that, \texttt{MH} is used. For both Metropolis-Hastings algorithms, the starting point is $\mu_{\text{start}}$. Because, \texttt{AMH} produces better results than the standard version, we only provide the results of \texttt{AMH}.
 
\paragraph{Wang-Landau} Another natural competitor, which was initially designed to explore target densities with several modes, is the Wang-Landau algorithm \citep{wang+l:2001}, whose convergence properties are studied in \cite{fort+j+k+l+s:2015}. The random walk version, denoted by \texttt{WL}, is when the proposal is $\phi_{\Sigma} $ with $\Sigma =  (0.4  / d) I_d $. The starting point is $\mu_{\text{start}}$. The adaptive version, which consists of a Robbins-Monro type of adaptation as documented in \cite{bornn+j+d+d:2013}, has been tried without improving the results compared to the non-adaptive case. The \texttt{PAWL} package was used to run the Wang-Landau algorithm (without parallel chains and using an initial chain to tune optimally the bins which are parameters of the algorithm).

\subsection{Results}

\subsubsection{Multimodal density}

We revisit the classical example introduced in \cite{cappe+d+g+m+r:2008} in which the target density is 
a mixture of two Gaussian distributions. Let
\begin{align*}
f(x) = .5 \phi_{\Sigma} (x- \mu) + .5 \phi_\Sigma (x + \mu),\qquad x\in \mathbb R^d,
\end{align*}
with $\mu  = (1,\ldots, 1) ^T / (2\sqrt d)$ and $\Sigma = (0.4  / d) I_d $. Note that the Euclidean distance between the two mixture centers is independent of the dimension as it equals $1$. The objective here is to recover both components of the mixture and this is clearly getting more difficult in large dimensions. 

To measure the accuracy of the algorithms, we compute the squared Euclidean distance between the estimated mean and the true mean. For each method, we take $\mu_{\text{start}} = (1,-1,0,\ldots, 0)^T/ \sqrt d$ as starting point. Such a choice for the initial sampler is to prevent the algorithms to take advantage from a very good start as is the case when $\mu_{\text{start}}  $ is $0$. We consider different values for the dimension $d$, namely $\{2,4,8, 12\}$, and all the algorithms are compared using a budget going from $5.10^4$ to $2.10^5$ evaluations of $f$.

The results are presented in Figure \ref{figu:ex1_mixture} (except for the case $d=2$ which was similar to $d =4$). As expected, \texttt{AMH} gives poor results compared to the ones of \texttt{SAIS} and \texttt{WL}. This is because \texttt{AMH} can hardly leave a mode. Among the two other methods, after $2.10^5$ requests to $f$, the three \texttt{SAIS} methods give similar results. Compared to \texttt{WL}, the squared error of \texttt{SAIS} is reduced by at least a factor $10$ in every dimension and even more than that in dimension $12$.

\begin{figure}
\centering
\includegraphics[scale=.16]{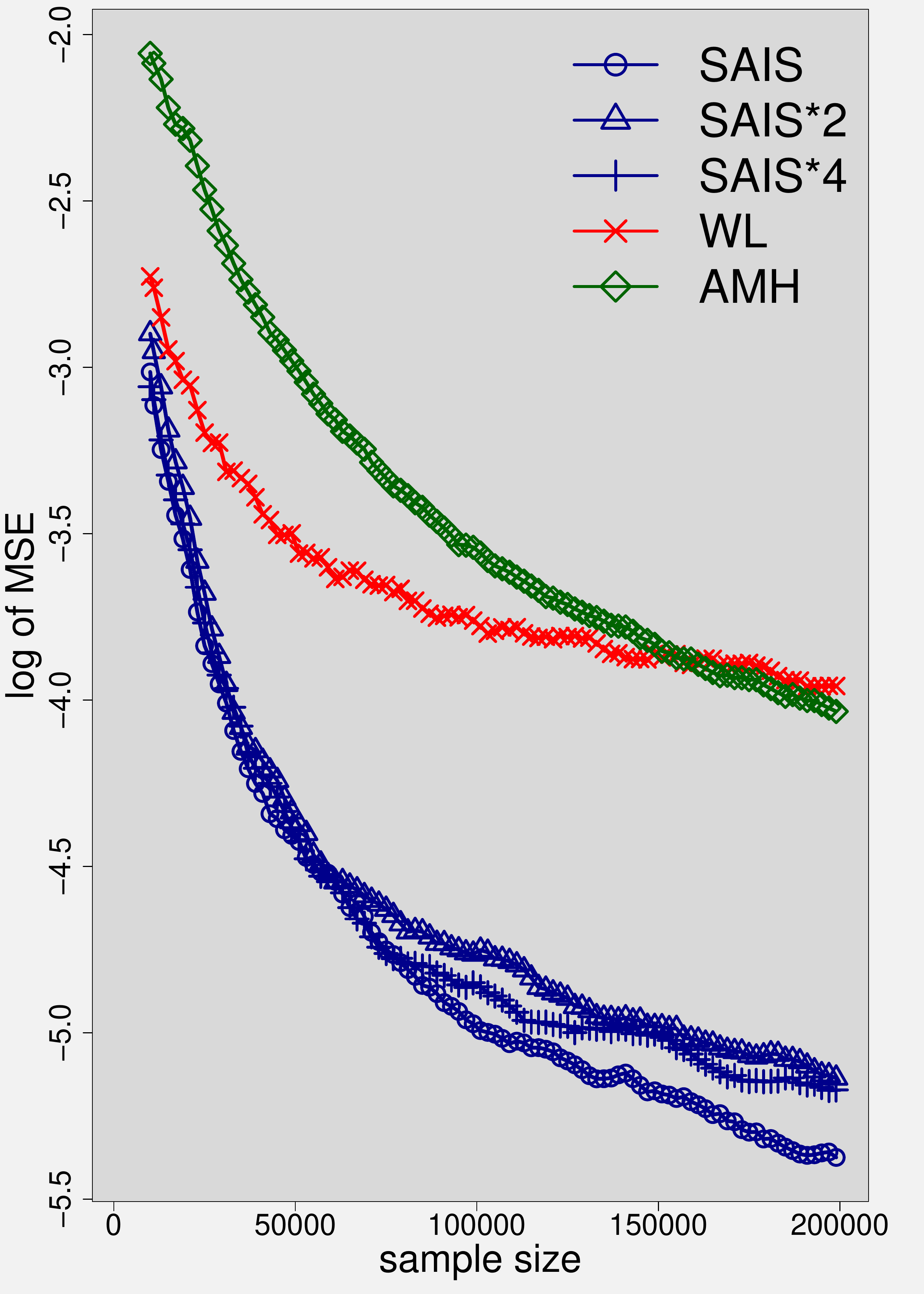}\includegraphics[scale=.16]{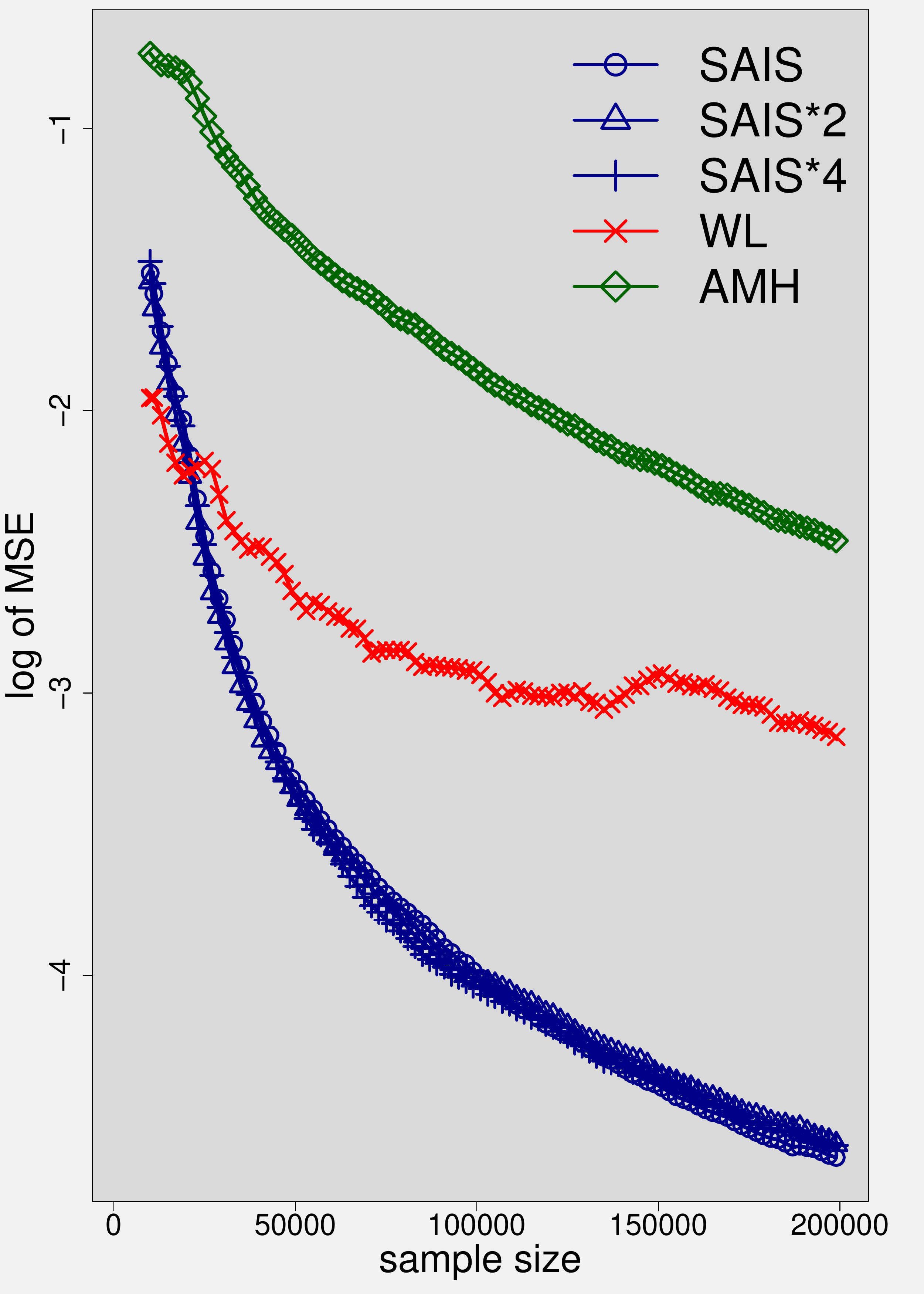}\includegraphics[scale=.16]{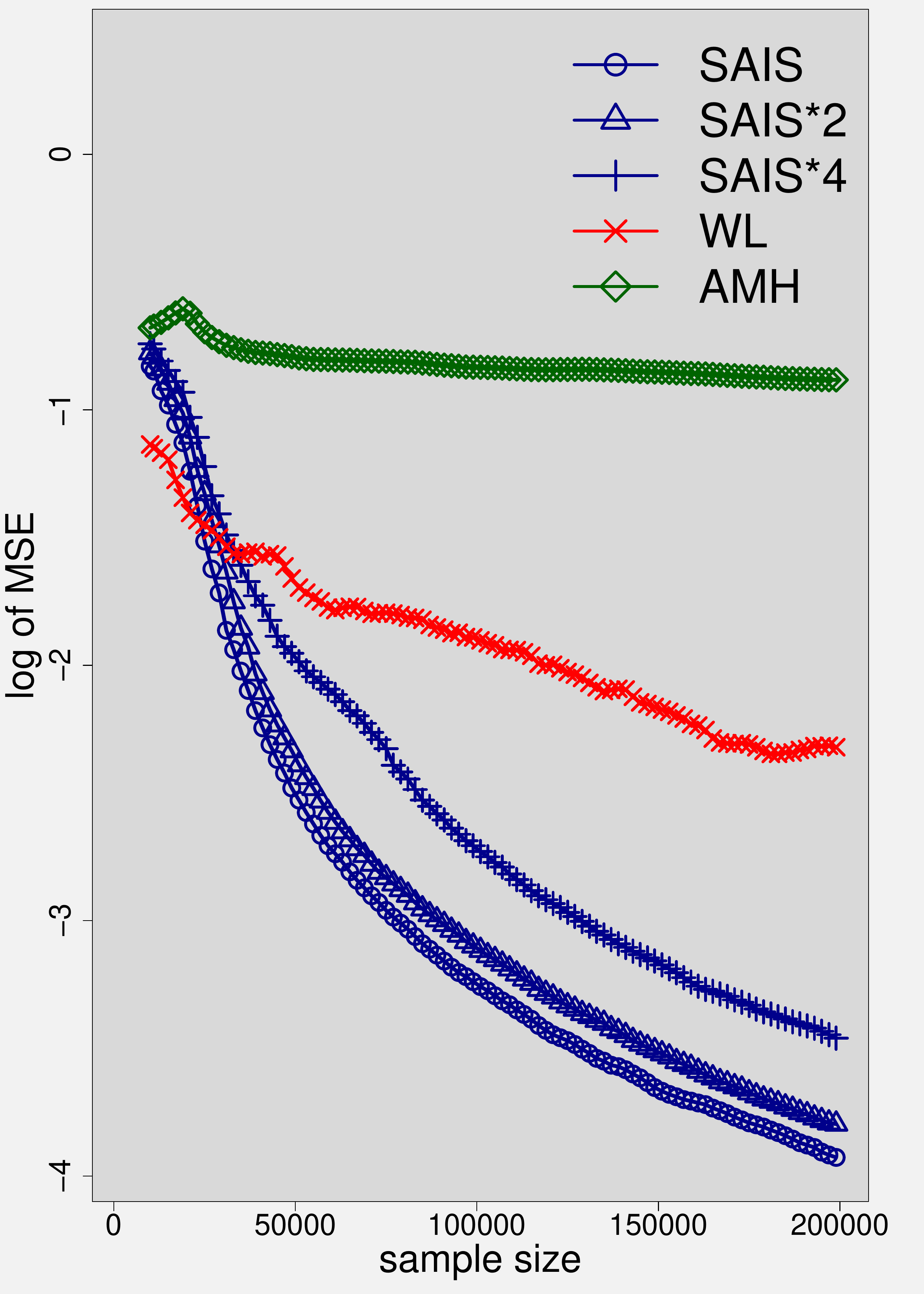}
\caption{\label{figu:ex1_mixture} (multimodal density) From left-to-right $d = 4, 8, 12$. Considered methods are described in the main text. Plotted is the median (based on $50$ independent replicates for each method) of the logarithm of the MSE with respect to the number of requests to the integrand.}
\end{figure}

\subsubsection{Cold start}

To conclude the simulation study, we illustrate the performance of \texttt{SAIS} when the starting 
distribution is far away from the target. In this example, the target distribution is given 
by $f(x)  = \phi_{\Sigma} (x - \mu)$ where $\mu = (5,\ldots, 5)^T /\sqrt d$ and  $\Sigma = (1/d) I_d$ 
whereas the initial distribution (for all the considered methods) has mean $\mu_{\text{start}} = 0$ and covariance $(4/d) I_d$. 
The main goal for the methods in competition is to converge rapidly around $\mu$.
The error is the same as before, the squared Euclidean distance between the true mean and the estimated mean. 
We consider different values for the dimension $d$, namely $\{2,4,8, 12\}$, and all the algorithms are compared using a budget going from $5.10^4$ to $2.10^5$ evaluations of $f$. 

The results are shown in Figure \ref{figu:ex2_cold_start} (except for $d=2$ which was similar to $d=4$). 
In this case, we observe a performance reversal between \texttt{AMH} and \texttt{WL} occurring at $d =8$. 
After that dimension, \texttt{WL} gives better results than \texttt{MH}. The improvement of \texttt{SAIS} compared to \texttt{AMH} and \texttt{WL} is substantial: the squared error is reduced by a factor $100$, in average.

\begin{figure}
\centering
\includegraphics[scale=.16]{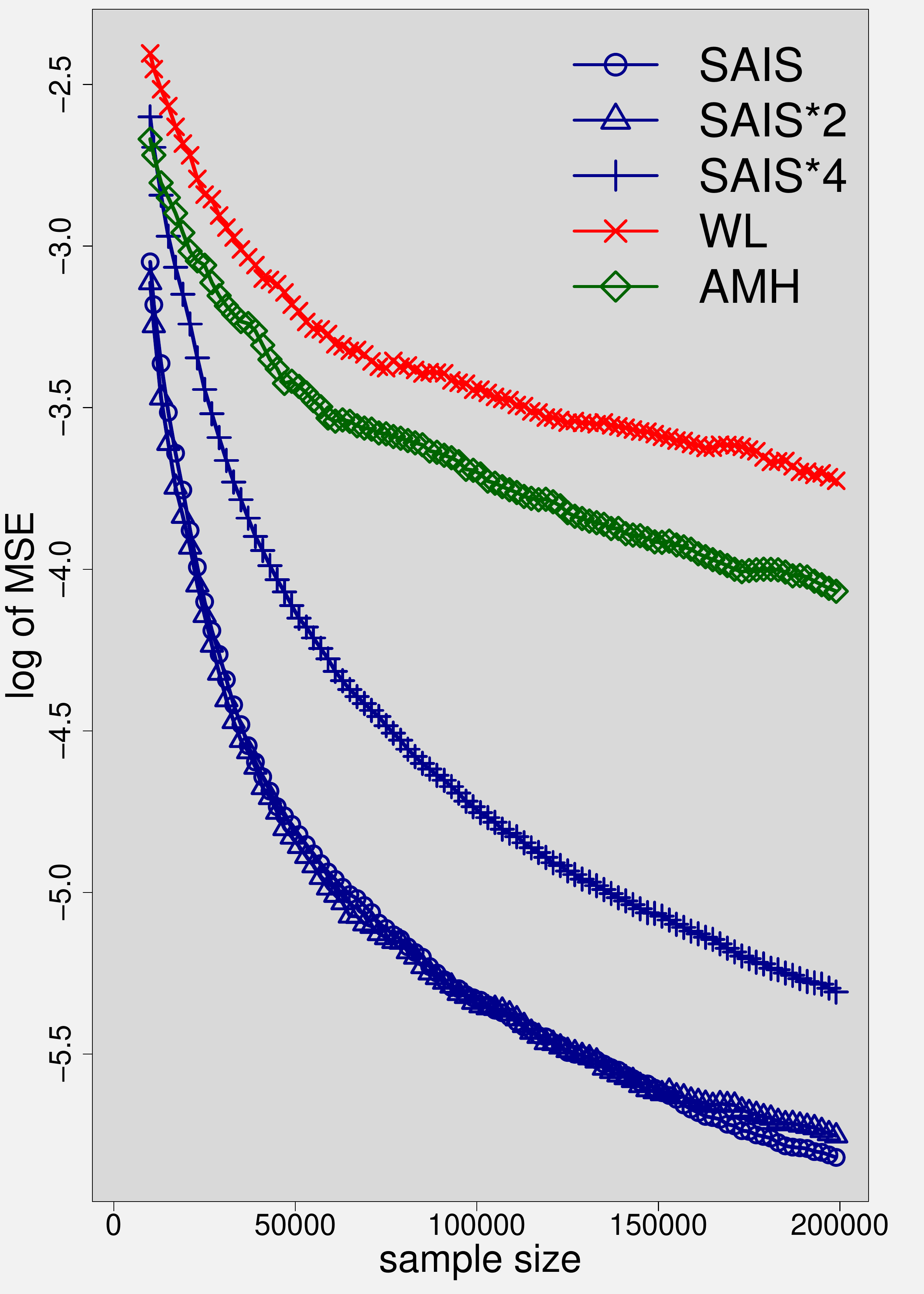}\includegraphics[scale=.16]{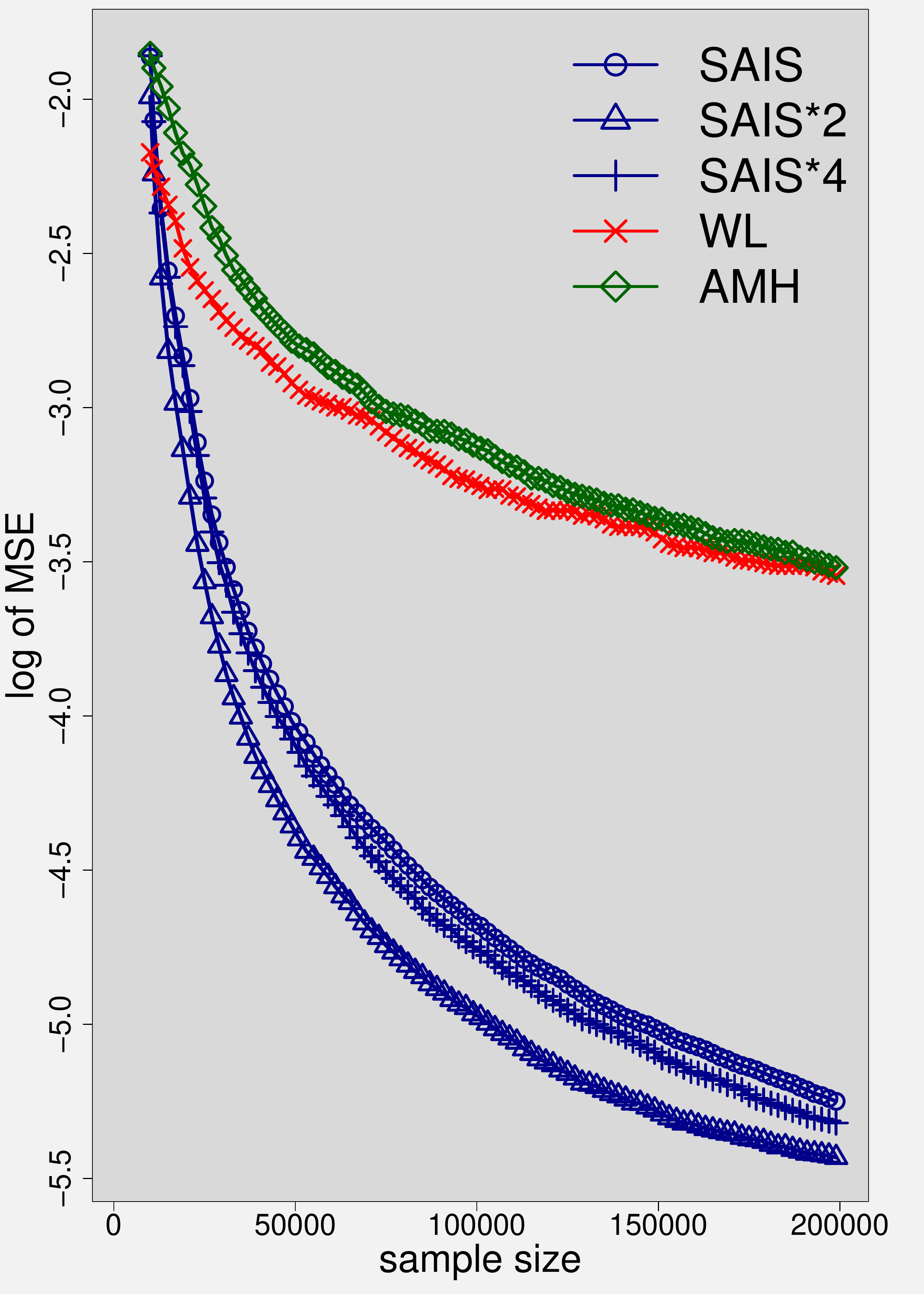}\includegraphics[scale=.16]{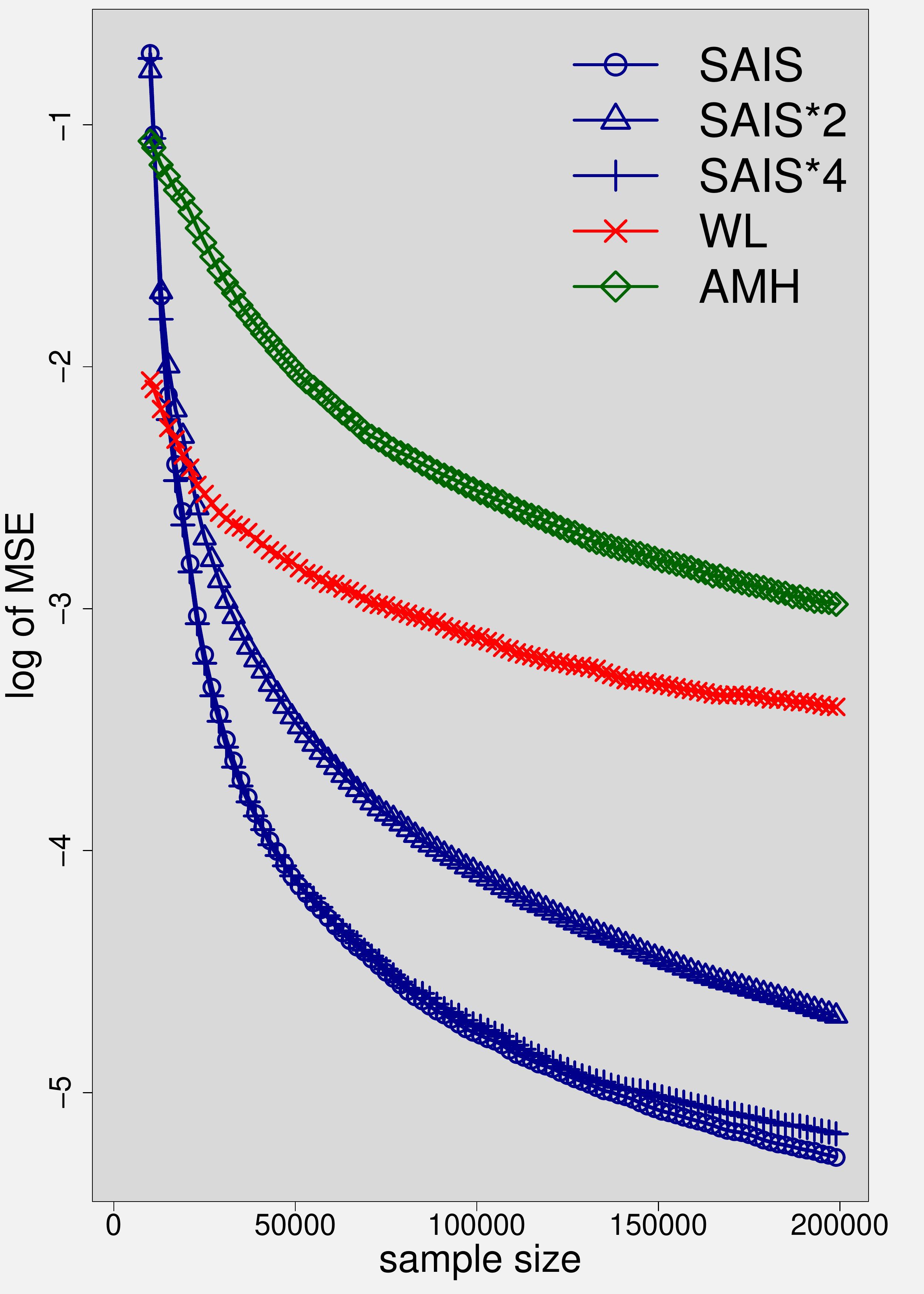}
\caption{\label{figu:ex2_cold_start} (cold start example) From left-to-right $d = 4, 8, 12$. Considered methods are described in the main text. Plotted is the median (based on $50$ independent replicates for each method) of the logarithm of the MSE with respect to the number of requests to the integrand.}
\end{figure}

\subsection{Computational efficiency}

At each stage $t= 1,\ldots, T-1$ of standard SAIS, the computing time needed to evaluate $q_{t}(X_i)$, $i\in B_{t+1} $ is $tm^2$. Consequently, running the first $t$ iterations represents $c_t = m^2 t(t-1) /2$ operations.
Concerning SAIS with subsampling, the situation is different as already discussed in Section \ref{sec:algo}. 
For the first $t$ iterations, an estimate of the number of operations required is then
$ m \sum_{k=1} ^ t \ell_k $, neglecting a log factor (due to multinomial sampling as detailed in Section \ref{sec:algo}). For the standard and the subsampling variants, the graphs representing the accuracy versus the computing time are provided in Figure \ref{fig:efficiency} in a logarithmic scale. We see a clear improvement given by the use of subsampling. For a similar error value, the overall computing time is almost $2$ times smaller in the logarithmic scale. This is a substantial gain as it means that in terms of computing time subsampling SAIS is approximately the square root of standard SAIS.

\begin{figure}
\centering
\includegraphics[scale=.16]{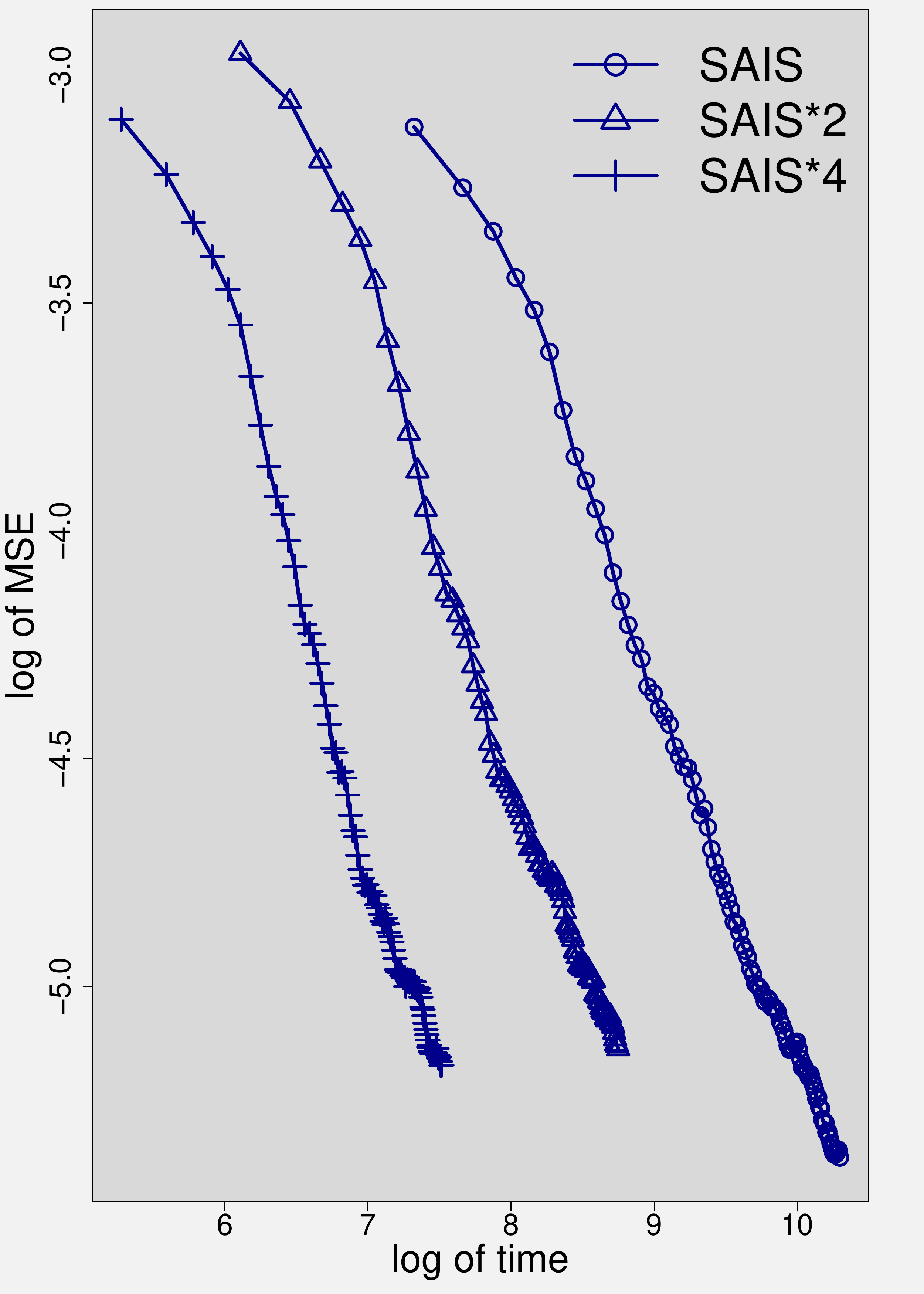}\includegraphics[scale=.16]{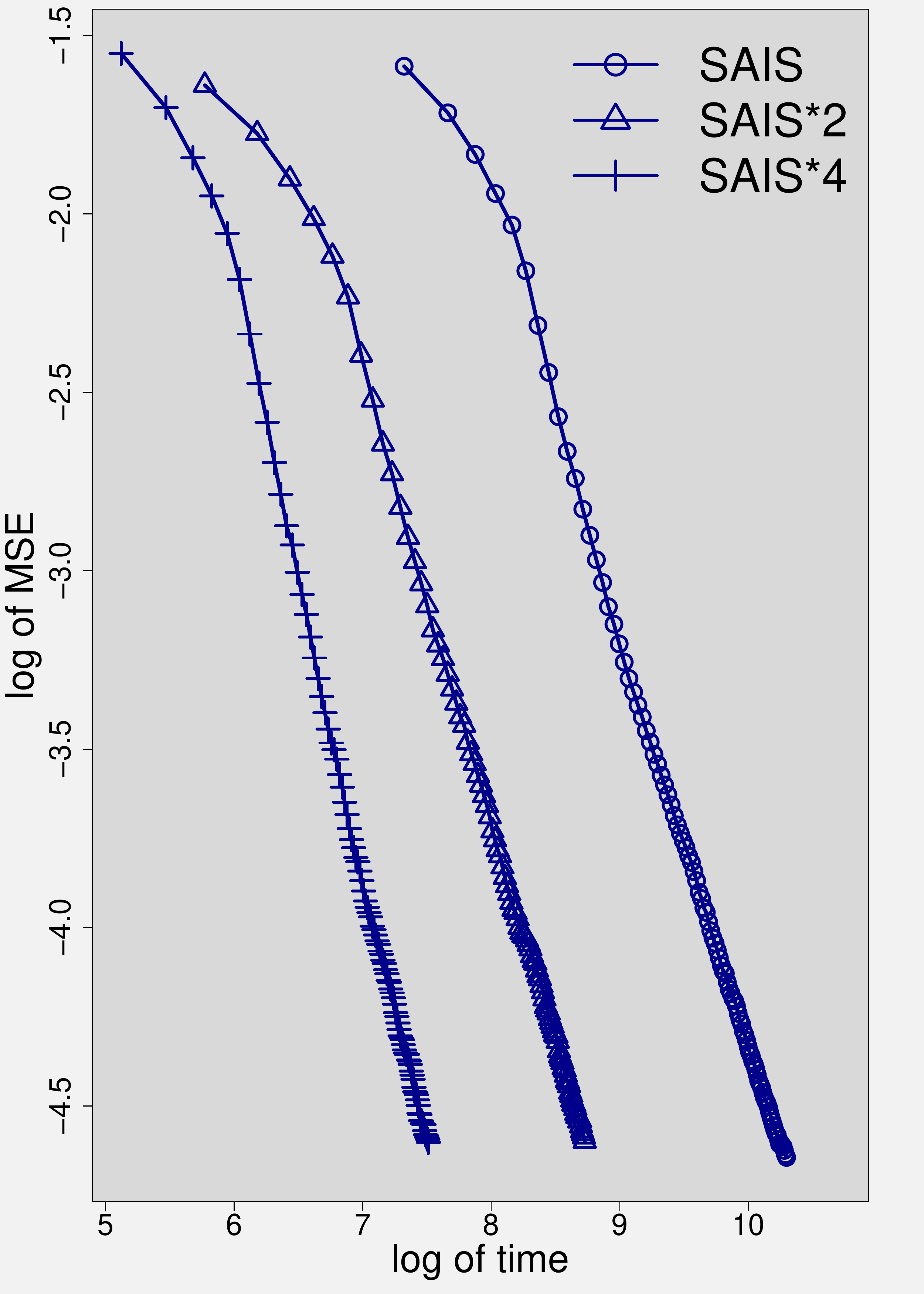}\includegraphics[scale=.16]{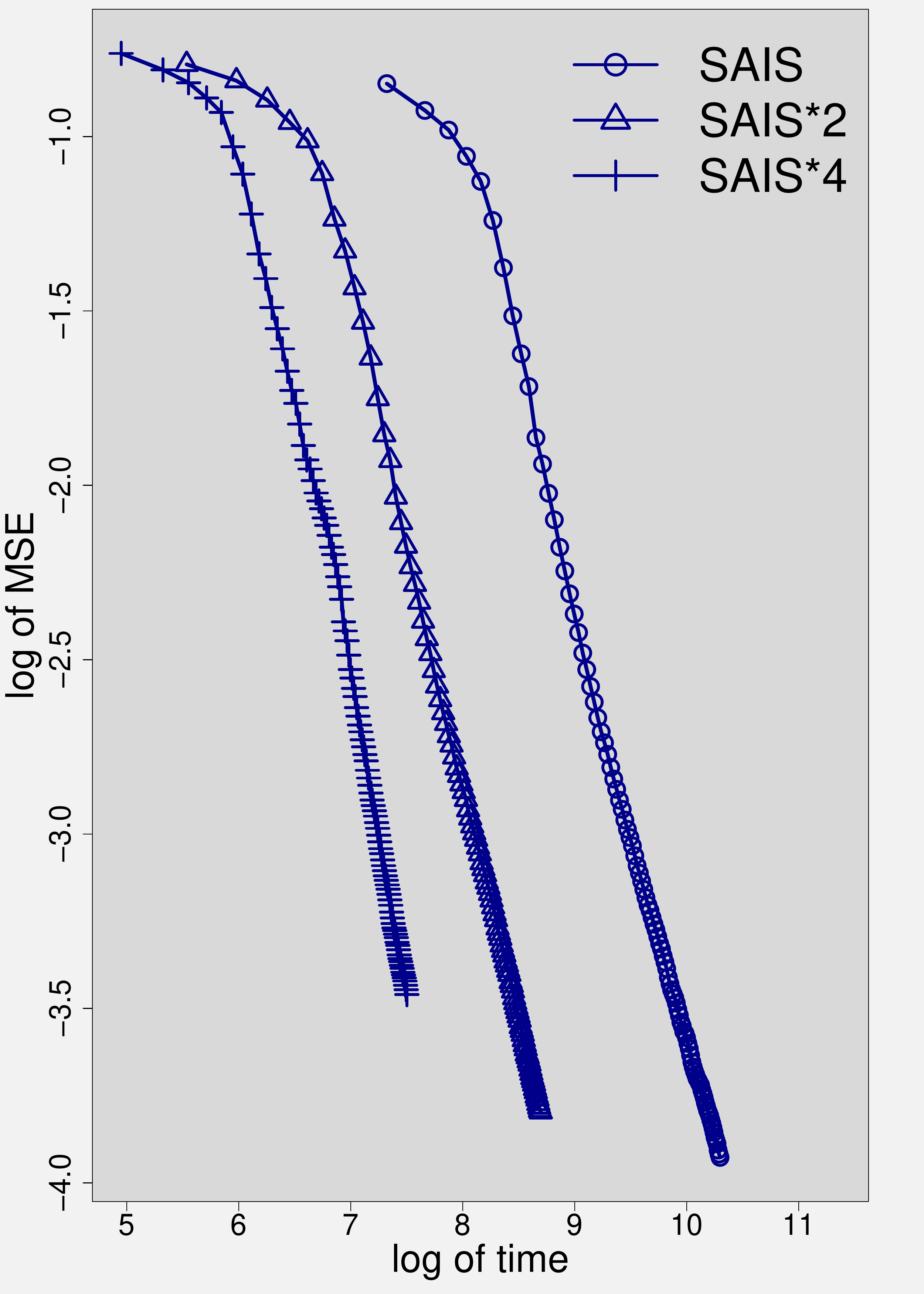}
\caption{\label{fig:efficiency} From left-to-right $d =  4, 8, 12$. Considered methods are described in the main text. Plotted is the median (based on $50$ independent replicates for each method) of the logarithm of the MSE with respect to the number of requests to the integrand.}
\end{figure}

\section{Mathematical proofs}\label{sec:proof}

All the claims of the theorems depend on $f_U(X_i)$ only through $W_{n,i}$, $i=1,\ldots, n$, which is independent of any normalizing constant. As a consequence we can assume without loss of generality that $f_U = f$ in the proofs.

\subsection{Proof of Lemma \ref{lemma_fixed_policy}}\label{append1}

We start with some preliminary remarks. Recall that for each $i\geq 1$ $W_i=f(X_i)/q(X_i)$, and that
\begin{align*}
&\tilde f_{n}  =  f *  K_{h_n},
\end{align*} 
and define 
\begin{align*}
&A_n(x) = n^{-1}  \Big(\sum_{i=1}^n W_i\Big) \left(f_n(x) - \tilde f_n(x)  \right) ,\\
&K_{i} = K_{h_n} (x-X_i) .
\end{align*}
By assumption, $(X_i)_{i\geq 1}$ is a sequence of independent random variables each having density $q$. Moreover, $\sigma_q^2 $ is continuous and bounded on $\mathbb R^d$ and integrable. The assumption on $K$ ensures that $\int K^2 <\infty$.  In particular, we have $$\mathbb E \left(\sum_{i=1}^n \{W_{i} - 1 \}\right)^2  = n \mathbb E \{W_{1} - 1 \} ^2 \leq n \int\sigma_q^2 .$$

Let $x\in \mathbb R^d$. From Slutsky Lemma, because $n^{-1} \left(\sum_{i=1}^n W_i\right) \to 1$ in probability, the proof will be complete as soon as we obtain that $ (nh_n^d)^{1/2} A_n(x) \leadsto \mathcal N ( 0,  \sigma_q^2(x) \int K^2)$.  We have
\begin{align*}
\sqrt{nh^d} \,
A_n(x) 
&=\sqrt{h_n^d/n} \, \left(\sum_{i=1}^n W_i\right) 
  \left( \sum_{i=1}^n W_{n,i}\big(K_{i}-\tilde f_n(x)\big) \right) \\
& = \sqrt{h_n^d/n} \,    { \sum_{i=1}^n W_{i}\big (K_{i}  - \tilde f_n (x)\big)}\\
& =  \sqrt{h_n^d/n} \, \left(  \sum_{i=1}^n
  \big\{ W_i K_{i}  - \tilde f_n (x)\big\}- \tilde f_n (x)\sum_{i=1}^n \{W_{i}-1\}\right),
\end{align*}
Using that $ \tilde f_n (x) \leq \sup_{x\in \mathbb R^d } f(x)$, we obtain that
\begin{align*}
\sqrt{nh_n^d} \, A_n(x) = \sqrt{h_n^d/n} \, 
   \sum_{i=1}^n \big\{ W_i K_{i}-\tilde f_n (x) \big\}    + O_P( h_n^{d/2} ).
\end{align*}
Invoking Slutsky Lemma again, and using that $ h_n\to 0$, it suffices to show that
\begin{align*}
\sqrt{h_n^d/n} \,  \sum_{i=1}^n \big\{W_{i} K_{i} - \tilde f_n (x) \big\}    
    \leadsto \mathcal N \Big( 0 , \sigma_q^2(x) \int K^2 \Big).
\end{align*}
Noting that
$\sum_{i=1}^n \{W_{i} K_{i}-\tilde f_n (x)\} = \sum_{i=1}^n \{W_{i} K_{i}-\mathbb E [W_i K_i] \}$, 
this convergence will be 
obtained by applying the Lindeberg central limit theorem \cite[Theorem 2.27]{vandervaart:1998}. 
We have to check the convergence of covariances  to $\sigma_q^2(x) \int K^2 $ and the Lindeberg condition. 
Because the elements of the previous sum are independent and centered, the covariance convergence means that
\begin{align*}
  h_n^d \big( \mathbb  E[W_{1}^2 K_{1}^2]- \mathbb  E [W_1 K_1] ^2 \big) 
   \to \sigma_q^2(x) \int K^2 .
\end{align*}
Recall that for any integrable and continuous function $g$, we have $(g\ast K_h) (x) \to g(x)$ as $h\to 0$, for all $x\in \mathbb R^d$ (see Section 8 of \cite{folland:2013}).
On the first hand, introducing the kernel $ \tilde K =  K^2 / \int K^2$, it holds that
\begin{align*}
h_n^d  \mathbb E[ W_{1}^2 K_{1}^2] 
&= h_n^d \int  \frac{f(y)^2}{q(y)} K \big((x-y)/h_n\big)^2 / h_n^{2d} \diff y \\
&=\Big(\int K^2\Big) ( \sigma_q^2 * \tilde K_{h_n})(x)  
\longrightarrow \sigma_q^2  (x)\int K^2.
 \end{align*}
On the other hand, one has $ \mathbb E [W_1 K_1]^2   = (f* K_h) (x) ^2  \to f(x)^2$, 
which implies that the second term in the variance is $O( h_n^d)$, negligible.
We finally need to verify the Lindeberg condition, i.e., for any $\epsilon>0$
\begin{align*}
\lim_{n\to \infty } \frac{h_n^d}{n } \sum_{i=1}^n\mathbb E\left[ W_i^2K_{i}^2
  1_{\{| W_i  K_i| >\epsilon \sqrt{n / h_n^{d}} \}}\right]   = 0.
\end{align*}
Because $W_i$ is bounded by $c$ and $K_i$ is bounded by $ h_n^{-d} \sup_{x\in \mathbb R^d} K(x)$ we have that
$ \{ | W_i  K_i| >\epsilon \sqrt{ n / h_n^{d}} \}  \subset  \{ C > \epsilon \sqrt{nh_n^d}\} $ for some $C>0$. But because $nh_n^d \to +\infty$, the previous set is empty for $n>n(\epsilon)$, implying the above statement.
\qed

\subsection{Proof of Lemma \ref{lemma:var_minimizer}}\label{append2}
If $q$ does not dominate $f$ then   $\int  f ^2 / q = +\infty$. If it does, using the Cauchy-Schwarz inequality, we obtain $ 1 = (\int f)^2 = (\int (f/ \sqrt q) \sqrt q )^2 \leq    \int  f ^2 / q = C(q)$. From this we deduce that $q= f$ is an argmin. If now $q$ is such that $\int  f ^2 / q = 1$, then equality holds in the Cauchy-Schwarz inequality meaning that $f  = \kappa q$ a.e. with $\kappa>0$. But $\kappa $ needs to be $1$ because $q$ and $f$ are densities.
\qed

\section{Proof of the preliminary results}\label{append:prelim}

Before entering the proofs of the preliminary results, let us introduce some notation related to the assumptions \ref{cond:f}, \ref{cond:f_tail} and \ref{cond:K}. As the function $f$, $K$ and $q_0$ are bounded, we denote by $U_f>0$, $U_K>0$ and $U_{q_0}>0$ their respective uniform bounds, i.e., for all $x\in \mathbb R^d$,
\begin{align*}
f(x)\leq U_f,\qquad K(x)\leq U_K, \qquad q_0(x)\leq U_{q_0}.
\end{align*}
 The Lipschitz constant of the function $K$ is denoted by $L_K$, that is, for all $(x,y)\in  \mathbb R^d \times \mathbb R^d$, 
\begin{align*}
&|K(x) - K(y) | \leq L_K \|x-y\|.
\end{align*}

\subsection{Proof of Lemma \ref{lemma:initial_consistency}}\label{append:initial}

Define
\begin{align*}
&Z_{n} (x) = \sum_{i=1}^n \big\{ W_i K_{h_n}(x-X_i ) - \tilde f_{n}  \big\},\\
&M_{n}  = \sum_{i=1}^n   \left\{  W_i - 1  \right\}.
\end{align*} 
The proof will follow from the use of $3$ independent lemmas which are now stated and proved.

\begin{lemma}\label{lemma:initial_variance}
Assume \ref{cond:f}, \ref{cond:f_tail}, \ref{cond:K} and work under the general policy \eqref{eq:general_policy}. If  $a_n^2  \ll \lambda_n$, then we have, for any $s>0$,
\begin{align}\label{Mnvit}
&|M_n| =O(\lambda_n^{-1/2} \sqrt{n\log n}) = o( n) ,\quad \text{a.s.}\\
\label{initZ} 
&\sup_{\|x\| \leq n^s } | Z_n(x)| = O (na_n\lambda_n^{-1/2}),\quad \text{a.s.}
\end{align}
Moreover, if for any $s>0$, with probability $1$,
\begin{align}\label{xpetit0}
&\sup_{\|x\|\leq n^s} | f_{n}(y) - \tilde f_{n}(y)  |=O\big(a_n\lambda_n^{-1/2}\big) .
\end{align} 
\end{lemma}

\begin{proof}
Equation \eqref{xpetit0} will follow from the decomposition
\begin{align}\label{df1} 
f_{n}(x)  - \tilde f_{n}(x) = \frac{Z_n(x) - \tilde f_{n}(x) M_n}{ M_n+n}.
\end{align}
Using \eqref{Mnvit} and \eqref{initZ}, combined with $\log(n) / (n\lambda_n) \to 0$
(since  $a_n^2  \ll \lambda_n$), we obtain that, almost surely,
\begin{align*}
&\sup_{\|x\|\le n^s}\big | f_{n}(x)-\tilde f_{n}(x)\big|=O\big(a_n\lambda_n^{-1/2}\big) + O\big(\lambda_n^{-1/2} \sqrt{\log(n) /n }  \big) = O\big(a_n\lambda_n^{-1/2}\big)  .
\end{align*}

The proof of (\ref{Mnvit}) follows from Theorem \ref{th:freedman}
 with 
$Y_i =   f(X_i) /q_{i-1}(X_i) - 1 $. Recall the definition of $c_\eta $ in \ref{cond:f_tail}. In particular, $f(x) \leq c_1 q_0(x) $ for all $x\in \mathbb R^d$. 
Note that, for any $i=1,\ldots, n$,
\begin{align}
\label{important_bound_1}&\sup_{x\in\mathbb R^d}\frac{f(x)}{q_{i-1}(x)} \leq c_1 \lambda_n^{-1}.
\end{align}
Since
$ |Y_i| \le{m} = \lambda_n^{-1}c_1 + 1$, and the quadratic variation is  not larger than
$v = n \lambda_n^{-1} c_1 $, we get
\begin{align*}
\mathbb P\big( |M_n|\geq t \big) \le 2\exp\left( -\frac{Ct^2\lambda_n }{n+t} \right),
\end{align*}
for some constant $C>0$ depending on $f$ and $q_0$. We conclude by taking $t = \gamma\sqrt{n\log n/\lambda_n}$ for $\gamma$ large enough
and use the Borel-Cantelli lemma.

We now consider (\ref{initZ}). Let us apply Corollary~\ref{th:freedmanZ} with  $\varepsilon=h_ n^{d+1}n^{-1}$, $\Omega_1=\Omega$, and 
\begin{align*}
&W_i =\frac{f(X_i)}{q_{i-1}(X_i)},\qquad\mu(dx)=f(x)dx, \qquad M(x) = c_1\lambda_n^{-1} .
\end{align*} 
It remains to choose $t$ and evaluate $m$, $\tilde v$ and $\tau$. 
We have, with the notation of Corollary~\ref{th:freedmanZ}
\begin{align*}
&m=U_K(1+c_1\lambda_n^{-1} ) h_n^{-d}\le C\lambda_n^{-1}h_n^{-d}\\
&v= n h_n^{-d} U_K \int M(u) K_{h_n}(x-u )\mu(du)
   \le C  nh_n^{-d}\lambda_n^{-1}\\
&\tau= 2L_K\varepsilon nh_n^{-d-1} =2L_K\\
&\tilde v=\max(v,2m\tau)\le Cnh_n^{-d}\lambda_n^{-1}
\end{align*} 
for some $C>0$.
Taking $t = \gamma\sqrt{n h_n^{-d}\lambda_n^{-1}\log n}$ for $\gamma\ge 1$, the conclusion of Corollary~\ref{th:freedmanZ} is
\begin{align*}
&\mathbb P\Big(\sup_{\|x\|\le n^s}|Z_n(x)|>t+\tau\Big)
  \le 4\Big(1+\frac{2n^{s+1}}{h_n^{d+1}}\Big)^d\exp\left(-\frac{t^2}{8(\tilde v+2mt/3)} \right).
\end{align*}
The term in the exponential is smaller than
\begin{align*}
-\frac{\gamma^2n\log n}{8Cn(1+\gamma\sqrt{ \log n/(n h_n^d\lambda_n)})}.
\end{align*}
Because $a_n^2=o (\lambda_n)$, the denominator is dominated by $ 8Cn$
and we get, since $t\ge\tau$ (for $n$ large enough), 
\begin{align*}
&\mathbb P\Big(\sup_{\|x\|\le n^s}|Z_n(x)|>2\gamma\sqrt{n h_n^{-d}\lambda_n^{-1}\log n}\Big)
  \le 4\Big(1+\frac{2n^{s+1}}{h_n^{d+1}}\Big)^d \exp\left(-C_1\gamma^2\log n\right)
\end{align*}
for some $C_1>0$. With $\gamma $ large enough, we obtain
\begin{align*}
\sum_{n\geq 1} \mathbb P \left(\sup_{\|x\|\le n^s} | Z_{n}( x)| 
  >2\gamma \sqrt{ n \log(n) / h_n^d \lambda_n} \right) < +\infty ,
\end{align*}
which by the Borel-Cantelli lemma implies (\ref{initZ}).
\end{proof}

\begin{lemma}\label{lemma:bias}
Under \ref{cond:f} and \ref{cond:K}, it holds that 
$\sup_{x\in \mathbb R^d} \big|\tilde f_n(x)-f(x)\big| = O(h_n^2)$.
\end{lemma}
\begin{proof}
Write
\begin{align*}
\big|\tilde f_n(x)-f(x)\big|
&=\Big|\int\big(f(x+h_ny)-f(x)\big)K(y)dy\Big|\\
&=\Big|\int\big(f(x+h_ny)-f(x)-h_n\langle y,\nabla f(x)\rangle\big)K(y)dy\Big|\\
&\le \frac{h_n^2}2\sup_{x\in \mathbb R^d} |\nabla^2 f(x) |\Big(\int \|y\|^2K(y)dy\Big).\qedhere
\end{align*}
\end{proof}

\begin{lemma} \label{lemma:fnxgrand} 
Assume \ref{cond:f}, \ref{cond:f_tail}, \ref{cond:K} and work under the general policy \eqref{eq:general_policy}. If  $a_n^2 \ll \lambda_n$, then there exists $s_0>0$ large enough such that
\begin{align}\label{fnan}
&\sup_{\|x\|> n^{s_0}} f_n(x) = o(a_n),\quad a.s.,\\
\label{important_bound_0}
&\sup_{\|x\| >   n^{s_0}} f(x)=o(a_n).
\end{align}  
\end{lemma}

\begin{proof} 
Let us start with (\ref{fnan}). Because $a_n^2\ll \lambda_n $, one has $nh_n^{d} \to \infty$ and $\log(n)/n \ll \lambda_n  $. We write 
\begin{align*}
f_n (x) = \sum_{i=1}^nW_{n,i} K_{h_n} ( x-X_i   ) ( \mathds 1_{\{\|X_i\|\leq  n^{s_0}/2\}} + \mathds 1_{\{\|X_i\|>  n^{s_0}/2\}})
\end{align*}
and treat both term separately. Using that $\sum_{i=1}^n W_{n,i} = 1$ and \ref{cond:K}, we find for any $\|x\|>n^{s_0}$,
 \begin{align*}
\sum_{i=1}^nW_{n,i} K_{h_n} ( x-X_i   )\mathds 1_{\{\|X_i\|\leq  n^{s_0}/2\}} 
&\leq \sup_{ \|y\|\leq n^{s_0}/2} K_{h_n} ( x- y  )\\
&= C_K h_n^{-d} \sup_{\|y\|\leq n^{s_0} /2 }   (1 + \|x-y \| / h_ n ) ^{-r_K} \\
& =C_K h_n^{-d}  (1 +  n^{s_0}/(2 h_ n )  ) ^{-r_K}.
\end{align*}
 This bound being $ O ( h_n^{-d+ r_K }  n^{-s_0 r_K} )$ 
and $nh_n^{d} \to \infty$, one might choose $s_0$ to make it $ O(n^{-1/2}) = o(a_n)$. For the other term, start with the bound
\begin{align*}
\sum_{i=1}^n&W_{n,i} K_{h_n} ( x-X_i   )  \mathds 1_{\{\|X_i\|>  n^{s_0}/2\}}\leq U_K h_n^{-d} \sum_{i=1}^n W_{n,i}  \mathds 1_{\{\|X_i\|>  n^{s_0}/2\}}.
\end{align*}
Because $\log(n)/n \ll  \lambda_n  $, we can use (\ref{Mnvit}) to obtain, with probability $1$,
\begin{align*}
\sum_{i=1}^n W_{n,i} \mathds 1_{\{\|X_i\|>  n^{s_0}/2\}} &=\frac1{M_n+n}\sum_{i=1}^n\frac{f(X_i)}{q_{i-1}(X_i)} \mathds 1_{\{\|X_i\|> n^{s_0}/2\}}  \\
&\le O(1) (n\lambda_n  ) ^{-1} \sum_{i=1}^n \frac{f(X_i)}{ q_0(X_i)} \mathds 1_{\{\|X_i\|> n^{s_0}/2\}} .
\end{align*}
Notice that assumption \ref{cond:f_tail}, with $\varepsilon$ small, implies that for any $p\geq 1$, there exists $C_p>0$ such that
\begin{align}\label{toimp}
f(x)/q_0(x) = C_p (1 + \| x\|)^{-p}, \qquad \|x\| \to +\infty.
\end{align}
It follows that 
\begin{align*}
\sum_{i=1}^n\frac{f(X_i)}{ q_0(X_i)} \mathds 1_{\{\|X_i\|> n^{s_0}/2\}} 
&\leq  C_p  \sum_{i=1}^n  (1 + \| X_i\|)^{-p} \mathds 1_{\{\|X_i\|> n^{s_0}/2\}}\\
&\leq  C_p  n  (1 +  n^{s_0}/2 )^{-p} .
\end{align*}
Since  $\lambda_n^{-1} h_n^{-d} \ll a_n^{-2} h_n^{-d}  \ll n $, 
we obtain with probability $1$,
\begin{align*}
\sum_{i=1}^nW_{n,i} K_{h_n} ( x-X_i   )  \mathds 1_{\{\|X_i\|>  n^{s_0}/2\}} = O ( \lambda_n^{-1} h_n^{-d}  n^{-ps_0} ) =o (  n^{1 -ps_0}). 
\end{align*}
Choose $p$ large enough to obtain that the previous bound is $o(n^{-1/2})= o (a_n)$.  Equation (\ref{important_bound_0}) is clearly a consequence of (\ref{toimp})
\end{proof}

\paragraph{End of the proof of Lemma~\ref{lemma:initial_consistency}}
Choose $s_0$ large enough so that the conclusion of Lemma~\ref{lemma:fnxgrand} is valid. We have
\begin{align}
\nonumber |f_{n}(x) - f(x)| & \leq  \sup_{\|x\|\leq n^{s_0}} |f_{n}(x) - f(x)|  + \sup_{\|x\|>n^{s_0}} f_n(x) + f(x)  \\
 \nonumber & \leq  \sup_{\|x\|\leq n^{s_0}} |f_{n}(x) - \tilde f_n (x)| + \sup_{x\in \mathbb R^d} |\tilde f_{n}(x) -  f(x)| \\
\label{decomp_final} & \qquad \qquad\qquad \qquad \qquad \qquad \qquad + \sup_{\|x\|>n^{s_0} } f_n(x) + f(x)  
\end{align}
Lemma~\ref{lemma:bias} implies that the middle term is $O(h_n^2)$. Lemma \ref{lemma:fnxgrand}  implies that the right-hand side term is $o(a_n)$ almost surely. Lemma \ref{lemma:initial_variance} implies that the left-hand side term is $O\big(a_n\lambda_n^{-1/2}\big) $.\qed

\subsection{Proof of Lemma~\ref{lemma:1}}\label{append:th:1}

The proof is similar to the proof of Lemma~\ref{lemma:initial_consistency}. The only difference is that, instead of using the bound \eqref{xpetit0}, we shall use an improved bound on $\sup_{\|x\|\le n^s} | f_{n}(x)-\tilde f_{n}|$ which will follow from the additional assumption (\ref{qfhyp}). This bound will be given in Lemma \ref{lemma:true_rate} but before we need to obtain the following technical lemma which will be used several times in the sequel.

\begin{lemma} \label{lemma:fqi}
Assume \ref{cond:f}, \ref{cond:f_tail}, \ref{cond:K} and work under the general policy \eqref{eq:general_policy}. If  $a_n^2 \ll \lambda_n $, then we have for any $\varepsilon \in (0,1/2)$
\begin{align}
\label{important_bound_2}
&\sup_{i\geq 0} \sup_{x\in\mathbb R^d} \lambda_i^{-\varepsilon} f(x)\Big|\frac{f(x)}{q_i(x)}-1\Big|
= O(1),\quad a.s.\\
\label{important_bound_3}
&\sup_{i\geq 0}\sup_{x\in\mathbb R^d}\frac{f^2(x)}{q_i(x)q_0(x)}<+\infty,\quad a.s.
\end{align}
\end{lemma}
\begin{proof} 
Set
\begin{align*}
d_i = \sup_{x\in \mathbb R^d}| q_i(x)- f(x)|.
\end{align*}
Since $f$ and $q_0$ are bounded, 
(\ref{qfhyp}) implies clearly that $d_i=O(\lambda_i^{1/2+\varepsilon})$ a.s.
If $f(x)\le 2d_i$, \ref{cond:f_tail} implies, for $i$ large enough,
\begin{align*}
f(x)\le f(x)^{1-\varepsilon} c_\epsilon q_0(x)
  \le (2d_i)^{1-\varepsilon} c_\epsilon q_i(x)/\lambda_i
  \le C_1\lambda_i^{-\frac12}q_i(x),
\end{align*}
for some $C_1 >0$ (because $\varepsilon\le 1/2$). 
If $f(x)>2d_i$, we have $q_i (x) >d_i$ but $ f(x)-q_i(x)\le d_i$ implying that $f(x)\le 2q_i(x)$,  
thus $f(x) \leq 2 \lambda_i^{-\frac12}q_i(x)$. 
In any case $f(x) \leq (C_1+2) \lambda_i^{-\frac12}q_i(x)$ and this leads to the bound
\begin{align*}
f(x)|f(x)-q_i(x)| \le C_1' \lambda_i^{-\frac12}q_i(x)
\lambda_i^{\frac12+\varepsilon}
\end{align*}
which implies (\ref{important_bound_2}). For (\ref{important_bound_3}), notice that because the $q_i$ are uniformly bounded (by $U_f+\sup_{i} d_i$), one has that
$ q_0^{1/2+\epsilon} q_i^{  1/2 -   \epsilon} $ is uniformly bounded by some $C_2>0$. Then, since $d_i=O(\lambda_i^{1/2+\varepsilon})$,
\begin{align*}
f(x)\le q_i(x)+C_3\lambda_i^{\frac12+\varepsilon}
\le q_i(x)+C_3\Big(\frac{q_i(x)}{q_0(x)}\Big)^{\frac12+\varepsilon} \le ( C_2  + C_3 ) \Big(\frac{q_i(x)}{q_0(x)}\Big)^{\frac12+\varepsilon},
\end{align*}
for some $C_3>0$. Now
\begin{align*}
f(x)^2
=f(x)^{(\frac12+\varepsilon)^{-1}}f(x)^{\frac{4\varepsilon}{1+2\varepsilon}}
\le C_4 \frac{q_i(x)}{q_0(x)}q_0(x)^2,
\end{align*}
where $C_4$ is given by using  \ref{cond:f_tail}. This is (\ref{important_bound_3}).
\end{proof}

\begin{lemma}\label{lemma:true_rate}
Assume \ref{cond:f}, \ref{cond:f_tail}, \ref{cond:K} and work under the general policy \eqref{eq:general_policy}. If $a_n=O(\lambda_n)$, we have for any $s>0$, with probability $1$,
\begin{align}\label{trueZ}
&\sup_{\|x\|\leq n^s}\,|Z_n(x)| =O\left(\sqrt{\frac{n\log(n)}{h_n^d}} \right) \\
&| M_n|=O\big(\sqrt {n\log n}\big).\label{trueM} 
\end{align} 
Moreover, for any $s>0$, with probability $1$,
\begin{align}\label{xpetit}
&\sup_{\|x\|\leq n^s} | f_{n}(y) - \tilde f_{n}(y)  |=O(a_n).
\end{align} 
\end{lemma}

\begin{proof}

Equation (\ref{df1}), namely $ f_{n}(y)  - \tilde f_{n}(y) = (Z_n(y) - \tilde f_{n}(y) M_n) /( M_n+n)$, implies that  (\ref{xpetit}) follows from \eqref{trueZ} and \eqref{trueM}.
%

In order to prove (\ref{trueZ}), we will use Corollary~\ref{th:freedmanZ} with
\begin{align*}
\Omega_1=  \Omega_U = \left\{\omega:  \sup_{i\geq 0}\sup_{x\in\mathbb R^d}\frac{f^2(x)}{q_i(x)q_0(x)} \leq U \right\}
 \end{align*}
where $U$ is a parameter, and
\begin{align*}
& W_i = \frac{f(X_i)}{  q_{i-1} (X_i)} ,\quad   \mu(dy)=f(y)dy, \quad M(x) =  \left(U \frac{q_0(x)}{f(x)} \right)\wedge  \left(c_1\lambda_n^{-1}\right).
\end{align*}
 It remains to choose $t$ and $\varepsilon$ and to evaluate $m$, $\tilde v$ and $\tau$.
As in the proof of Lemma~\ref{lemma:initial_variance}, 
we take $\varepsilon=h_n^{d+1} n^{-1}$, and  
\begin{align*}
&m=U_K(1+U_w) h_n^{-d}\le U_K(1+c_1 \lambda_n^{-1} )  h_n^{-d} \\
&v=U_Kn h_n^{-d}\int M(u) K_{h_n}(x-u ) f(u) du \leq U_Kn h_n^{-d} U_{q_0} U \\
&\tau=  2L_K\varepsilon_n nh_n^{-d-1} =2L_K\\
&\tilde v=\max(v,2m\tau).
\end{align*} 
The conclusion of Corollary~\ref{th:freedmanZ} is that, for all $t\geq 0$,
\begin{align*}
&\mathbb P\Big(\sup_{\|x\|\le n^s}|Z_n(x)|>t+\tau,\,\Omega_U\Big)
  \le 4(1+2n^s/\varepsilon)^d \exp\left(-\frac{t^2}{8(\tilde v+2mt/3)} \right).
\end{align*}
We take $t=\gamma \sqrt{n\log(n)h_n^{-d}}$, and notice that $t\ge\tau$ for $n$ large enough. Using that $\lambda_n^{-1} = O(a_n^{-1})$, we find that $mt \leq C \gamma \lambda_n^{-1} h_n^{-d}\sqrt{n\log(n)h_n^{-d}} \le C\gamma nh_n^{-d}$ and also $\tilde v\le Cnh_n^{-d}$, when $n$ is large enough, for some $C>0$.
Taking $\gamma$ larger than $3/2$, we find
\begin{align*}
\mathbb P\Big(\sup_{\|x\|\le n^s}|Z_n(x)|>2t,
\,\Omega_{U}\Big) &  \le 4\Big(1+\frac{2n^{s+1}}{h_n^{d+1}}\Big)^d \exp\left(-\frac{3t^2}{32 C \gamma nh_n^{-d}}\right)\\
&\le 4\Big(1+\frac{2n^{s+1}}{h_n^{d+1}}\Big)^d \exp\left(- 3 \gamma\log (n) / (32C) \right).
\end{align*}
For $\gamma$ large enough, the Borel-Cantelli lemma applies and one has for any $U>0$ that 
\begin{align*}
\mathbb P \left( \limsup_{n\to \infty} \sup_{\|x\|\le n^s}\frac{|Z_n(x)|}{\sqrt{n\log(n)h_n^{-d}}} >\gamma , \,\Omega_U \right)  =0
\end{align*}
which implies the result since,  by Lemma~\ref{lemma:fqi}, $\mathbb P(\Omega_U)\to 1$ as $U\to +\infty$.

The proof of (\ref{trueM}) will use Theorem \ref{th:freedman}.
The bounds on the quadratic variation and the increments of $M_n$ are derived in a very 
similar way as before. By (\ref{important_bound_1}) and $\lambda_n^{-1} = O(a_n^{-1})$, we have
\begin{align*}
& \max_{i=1,\ldots, n} \sup_{y\in \mathbb R^d} \left| \frac{f(y)  }{ q_{i-1}(y)} - 1 \right| \leq    c_1 \lambda_n^{-1}  +1 \leq  C\Big(\frac n{\log n}\Big)^{1/2} =m\\
&\sum_{i=1}^n \int \frac{ f(y)^2 }{q_{i-1}(y)} \,\diff y\leq  nU  = v\quad\text{on $\Omega_U$}.
\end{align*}
Theorem \ref{th:freedman} implies that for any $t>0$
\begin{align*}
\mathbb P \left(|M_n| >t, \Omega_U\right) 
& = \exp\Bigg( -\frac{ t^2}{ 2(nU+(C/3) t \sqrt{n/\log n} )} \Bigg).
\end{align*}
Choosing $t= \gamma \sqrt {n\log n}$ with $\gamma$ large leads to the summability
of these probabilities and (\ref{trueM}) holds on $\Omega_{U}$. 
Since $\mathbb P(\Omega_{U})\to 1$ as $U\to\infty$ (Lemma~\ref{lemma:fqi}), we get (\ref{trueM}).
\end{proof}

\subsection{Proof of Lemma~\ref{lemma:2}}\label{append:lemma:2}

We rely on the following central limit theorem for martingale arrays.

\begin{theorem}\cite[Corollary 3.1]{hall+h:1980}\label{th:hall}
Let $(w_{n,i})_{1\leq i\leq n,\, n\geq 1} $ be a triangular array of random variables such that
\begin{align}\label{hh1}
&\mathbb E [w_{n,i}\mid \mathcal F _{i-1}] = 0,\quad \text{for all }1\leq i \leq n,\\
&\sum_{i=1}^ n \mathbb E[ w_{n,i}^2\mid \mathcal F _{i-1} ] \to v^*\geq 0,\quad \text{in probability,}\label{hh2}\\
&\sum_{i=1}^ n \mathbb E[ w_{n,i}^2 \mathrm I _ {\{ |w_{n,i}| > \varepsilon \} } \mid \mathcal F _{i-1} ] \to 0 ,\quad \text{in probability,}\label{hh3}
\end{align}
then, $\sum_{i=1}^ n w_{n,i} \leadsto \mathcal N (0,v^*) $,  as $ n\to \infty$.
\end{theorem}

\paragraph{Proof of equation \eqref{gnorm}} Set
\begin{align*}
 I_n(g) = n^{-1} \sum_{i=1} ^n W_ig(X_i)
\end{align*}
Because $\log(n)/n\ll \lambda_n $, applying Lemma \ref{lemma:initial_variance} gives that $I_n(1)$ converges in probability to $1$. 
Using the decomposition
\begin{align*}
\sum_{i=1}^n W_{n,i} g(X_i ) -  I(g)  
  = \frac1{I_n(1)}\left\{\Big(I_n(g)-I(g)  \Big)-\big(I_n(1)-1\big)I(g)  \right\},
\end{align*}
with $I(g)  = \int gf$, the problem reduces to the estimation of the limit
of $n^{1/2}((I_n(g)- I(g)  ),(I_n(1)-1)) $. 
Verifying the conditions of Theorem \ref{th:hall} with $v_*=V(f,g)$ and
\begin{align*}
w_{n,i} =\frac1{\sqrt n} \Big(\frac{g(X_i) f(X_i) }{q_{i-1}(X_i)} - I(g)  \Big),
\end{align*}
will prove the convergence of $n^{1/2}(I_n(g)-I(g)  )$ to the limit $\mathcal N(0, V(f,g))$.
This will imply that $n^{1/2}(I_n(1)-1)$ converges to zero in probability since $V(f,1)=0$,
and the result will follow by virtue of Slutsky's Lemma. 

Equation (\ref{hh1}) is satisfied. We now show (\ref{hh2}) with $v_*=V(f,g)=\int g^2f-I(g) ^2$, or equivalently that
\begin{align}\label{tig2f}
n^{-1}\sum_{i=1}^n \int \frac{g ^2f ^2 }{ q_{i-1}}  \to \int g^2f  , \qquad \text{in probability.}
\end{align}
But using (\ref{important_bound_2}),  we have that, with probability $1$,  
for any $x\in \mathbb R^d$, 
$$| g (x)^2f(x) ^2 / q_{i-1}(x)-g(x)^2f(x)|\to 0.$$ 
Moreover, using \eqref{important_bound_3}, we have that with probability $1$,  
$g ^2f ^2 / q_{i-1}\leq C g^2 q_0$ for some $C>0$, the integral of which is finite. 
By the Lebesgue dominated convergence theorem we find that, almost surely,
\begin{align*}
\int \Big|\frac{g ^2f ^2 }{ q_{i-1}}  - g^2f\Big| \to 0,
\end{align*}
and so does the average from the Cesaro lemma. This implies (\ref{tig2f}).

Finally, we verify the Lindeberg condition (\ref{hh3}). We have to prove that
\begin{align*}
n^{-1}\sum_{i=1}^n \int \left(\frac{g(y)f(y)}{ q_{i-1}^2(y)}- I(g) \right)^2q_{i-1} (y) \mathrm I_{A_{ni}}(y) \,\diff y \to 0~~~~ \text{in probability},
\end{align*}
with $A_{ni}=\big\{|gf / q_{i-1}- I (g) | >\varepsilon \sqrt n \big\}$.
For $n$ large enough
 $$A_{ni}\subset A_{ni}' =  \big\{gf / q_{i-1} >\varepsilon \sqrt n/2 \big\}\subset A_n := \big\{gf / q_0  >  \varepsilon  \lambda_ n \sqrt n/2 \big\}  $$
for all $i= 1,\ldots, n $. 
Consequently, using \eqref{important_bound_3} and \eqref{qfhyp}, we get, a.s.,
\begin{align*}
&n^{-1}\sum_{i=1}^n \int \left(\frac{g(y)f(y)}{ q_{i-1}^2(y)}- I(g) \right)^2q_{i-1} (y) \mathrm I_{A_{ni}}(y) \,\diff y \to 0 \\
&\leq 2 n^{-1}\sum_{i=1}^n \left(\int \frac{g(y)^2f(y)^2 }{ q_{i-1} (y)} \mathrm I_{A_{ni}}(y) \,\diff y  +  I(g)^2  \int q_{i-1} (y) \mathrm I_{A_{ni}}(y) \,\diff y\right)\\
&\leq 2 n^{-1}\sum_{i=1}^n \left(\int \frac{g(y)^2f(y)^2 }{ q_{i-1} (y)} \mathrm I_{A_{n}}(y) \,\diff y  +  I(g)^2  \int q_{i-1} (y) \mathrm I_{A_{ni}'}(y) \,\diff y\right)\\
& \leq 2 n^{-1}\sum_{i=1}^n \left( C \int  g(y)^2 q_0(y) \mathrm I_{A_{n}}(y) \,\diff y  +  2 I(g)^2  ( \varepsilon \sqrt n)^{-1}  \int  g(y)f(y)  \diff y\right)\\
&=  2  \left( C \int  g(y)^2 q_0(y) \mathrm I_{A_n}(y) \,\diff y  +  2 I(g)^2  ( \varepsilon \sqrt n)^{-1}  \int  g(y)f(y)  \diff y\right).
\end{align*}
Because $\int  g ^2 q_0 < \infty$, applying the Lebesgue dominated convergence theorem, we find that the limit is $0$. \qed

\section{Proof of Theorem \ref{th:cvd3}}\label{app:proof_cvd3}

The study of SAIS with subsampling starts with the following result which provides a uniform bound on the error associated to the bootstrap density estimate. Because the error rate below is $b_n+ a_n\lambda_n^{-1/2}+h_n^2$, the assumption allows to check that $\eqref{qfhyp}$ with $f^* = f^\circ$ and then to apply Lemma \ref{lemma:2} to obtain equation \eqref{gnorm} which is the stated result.

\begin{theorem}\label{bootconv}
Assume \ref{cond:f}, \ref{cond:f_tail}, \ref{cond:K} and work under the subsampling policy \eqref{sampler_update_boot}. If $a_n^2 \ll \lambda_n$ and $\ell_n$ is such that
\begin{align}\label{hypmn}
\ell_n\gg \frac{\log(n)}{h_n^d}, 
\end{align}
then,
\begin{align*}
&\sup_{x\in\mathbb R^d}|f_{n}^*(x)-f(x)|
  =O\big(b_n+ a_n\lambda_n^{-1/2}+h_n^2\big), \qquad \text{a.s.}
\end{align*}
with $b_n=\big(\log n/(\ell_nh_n^d)\big)^{1/2}$.
\end{theorem}
\begin{proof}
Recall that
\begin{align*}
f^*_{n} (x) = \ell_n^{-1} \sum_{i=1} ^ {\ell_n} K_{h_n} (x - X_{n,i}^*),\qquad x\in \mathbb R^d,
\end{align*}
where $(X_{n,i}^*)_{i=1,\ldots, n} $ are independent and identically distributed random 
variables with distribution $\mathbb P_n = \sum_{i=1} ^n W_{n,i} \delta_{X_i}$, 
conditionally on $X_1,\ldots, X_n$.
The decomposition is as follows
\begin{align*}
f^*_{n} (x)  - f(x) = \{f^*_{n} (x)  - f_n(x)\}  + \{f_n(x)  -  f(x)\}. 
\end{align*}
The second term in the right-hand side is treated by applying Lemma \ref{lemma:initial_consistency} with $f^\circ_n=f^*_n$. It gives
\begin{align*}
\sup_{x\in \mathbb R^d} |  f_n(x)  -  f(x)| = O(a_n \lambda_n^{-1/2} + h_n ^2 ) . 
\end{align*}
For the first  term in the right-hand side, write
\begin{align*}
\sup_{x\in \mathbb R^d} |  f_n^*(x)  -  f_n(x)|  \leq \sup_{\|x\|\leq n^s} |  f_n^*(x)  -  f_n(x)|  + \sup_{\|x\|>n^s} | f^*_n(x)- f_ n (x)| 
\end{align*}
where $s>0$. Then showing that 
\begin{align}
&\label{f*petitx1} \sup_{\|x\|\leq n^s} |  f_n^*(x)  -  f_n(x)|  =  O( b_n ),~~~a.s.\\
&\label{f*grandx} \sup_{\|x\|>n^s}f^*_n(x)=o( b_n ) ,~~~a.s.
\end{align}
will be enough to conclude the proof.

We now show \eqref{f*petitx1}. Define 
\begin{align*}
Z_n^*(x) = \sum_{i=1}^{\ell_n} \{ K_{h_n} (x - X_{n,i}^*) - f_n(x)\} .
\end{align*}
Since  $\mathbb P_n = \sum_{i=1} ^n W_{n,i} \delta_{X_i}$ is the the conditional 
expectation given the initial 
sample $X_1,\ldots, X_{n}$, we have for all $i=1,\ldots, n$,
\begin{align*}
&\mathbb E_n[K_{h_n} (x - X_{n,i}^*)]  = f_n(x) .
\end{align*}
Let us apply Corollary~\ref{th:freedmanZ}, with $\Omega_1 = \Omega$, 
\begin{align*}
&w_i =1,\qquad\mu(dx)=\mathbb P_n=\sum_{i=1}^n W_{n,i}\delta_{X_i},\qquad M(x) = 1,
\end{align*} 
and get, with probability $1$,
\begin{align*}
&\mathbb P \Big(\sup_{\|x\|\le n^s}|Z_n^*(x)|>t+\tau\mid \mathcal F_ n \Big)
  \le 4(1+2n^s/\varepsilon)^d \exp\left(-\frac{t^2}{8(\tilde v+2mt/3)} \right).
\end{align*}
It remains to choose $t$ and $\varepsilon$ and evaluate $m$, $\tilde v$ and $\tau$.
We take $\varepsilon=h_n^{d+1}\ell_n^{-1}$ and obtain
\begin{align*}
&m=2U_Kh_n^{-d}\\
&v=\ell_n h_n^{-d}U_K\int K_{h_n}(x-u )\mu(du) =  \ell_n h_n^{-d}U_Kf_n(x)\\
&\tau=2L_K\ell_n h_n^{-d-1}\varepsilon_n=2L_K\\
&\tilde v=\max(v,2m\tau)\le  C(1+f_n(x)) \ell_n h_n^{-d},
\end{align*}  
for some constant $C>0$. We take
\begin{align*}
&t=t_n= (m/3) \gamma \log n +\sqrt{\tilde v \gamma \log n}.
\end{align*}
With this choice, we get
\begin{align*}
&\frac{t_n^2}{8(\tilde v+mt_n/3)}\ge \frac{\max( t_n (m/3) \gamma\log n, \tilde v \gamma \log n)}{8(\tilde v+mt_n/3)}\ge \frac1{16}\gamma\log n
\end{align*}
Hence, with probability $1$,
\begin{align*}
&\mathbb P\Big(\sup_{\|x\|\le n^s }|Z_n^*(x)|>t_n+2L_K \mid \mathcal F_ n \Big)
  \le 4(1+2n^s/\varepsilon)^d n^{-\gamma/16}
\end{align*}
With $\gamma\ge 1$ large enough, taking expectation on both sides and applying 
the Borel-Cantelli lemma we get that a.s.
\begin{align*}
\limsup_{n} \sup_{\|x\|\le n^s}\frac{|Z_n^*(x)|}{t_n+2L_K}\le 1.
\end{align*}
Now, 
since $t_n\to\infty$, it suffices to check  
that $t_n/\ell_nb_n$ is bounded almost surely; but for a certain $C >0$,
\begin{align*}
\frac{t_n^2}{\ell_n^2b_n^2}
\le C\gamma^2\frac{h_n^{-2d}(\log n)^2+ (1+ f_n(x)) \ell _n h_n^{-d}\log n}{\ell_n h_n^{-d}\log n}
\le C\gamma^2( b_n^2+1+f_n(x))
\end{align*}
and we conclude because $b_n\to 0$ by assumption, and, with probability $1$, $f_n(x)$ is bounded by Lemma~\ref{lemma:initial_consistency}.

We now show \eqref{f*grandx}. Let $x\in \mathbb R^d$ be such that $\|x\|>n^s$. We have
\begin{align*}
f^*_n(x)
&=\frac1{\ell_n}\sum_{i=1}^{\ell_n} K_{h_n}(x-X^*_{n,i})1_{\|X^*_{n,i}\|\le n^{s}/2}
+ \frac1{\ell_n}\sum_{i=1}^{\ell_n} K_{h_n}(x-X^*_{n,i})1_{\|X^*_{n,i}\|> n^{s}/2}\\
&\le   \sup_{\|y\|\leq  n^{s}/2 } K_{h_n}( (x-y)/h_n)
+U_K \frac{h_n^{-d}}{\ell_n}\sum_{i=1}^{\ell_n} 1_{\|X^*_{n,i}\|> n^{s}/2}.
\end{align*}
If $s$ is large enough, the first term is $o(a_n)= o(b_n)$ as it was shown in the proof of Lemma \ref{lemma:fnxgrand}.
For the second term, it suffices to prove that
\begin{align*}
\xi_n=  \Big(\sum_{i=1}^nW_i\Big)  \frac{h_n^{-d}}{\ell_n}\sum_{i=1}^{\ell_n} 1_{\|X^*_{n,i}\|> n^{s}/2}\to 0~~~~a.s.
\end{align*}
because $\Big(\sum_{i=1}^nW_i\Big)/n\to 1$ a.s. by Lemma (\ref{lemma:initial_variance}) and $a_n^2 \ll \lambda_n $. 
Note that
\begin{align*}
\mathbb E_n \left[  \xi_n \right]=  h_n^{-d} \sum_{i=1}^{n} W_i 1_{\|X_i\|> n^{s}/2}.
\end{align*}
Thus
\begin{align*}
\mathbb E[\xi_n]=  n h_n^{-d}\int f(x)1_{\|x\|> n^{s}/2}dx.
\end{align*}
Since $f$ decreases faster than any polynomial, the integral can be bounded by $C_Kn^{-K}$
for any $K>0$, implying that $\mathbb E[\sum \xi_n]<\infty$, thus $\xi_n\to 0$ a.s.
\end{proof}

%

\section{Bernstein inequalities for martingale processes}\label{append_freedman}
This section is devoted to the derivation of Corollary~\ref{th:freedmanZ} below which plays a crucial role in the study of the kernel density estimate. 
Everything is based a Bennett inequality for martingales given by Freedman in 1975, that we will first modify in order to 
allow the martingale increments to be unbounded.

\begin{theorem}\label{th:freedman}
Let $(\Omega, \mathcal F , (\mathcal F _{i})_{i\geq 1}, \mathbb P)$ be a filtered space. Let $(Y_i)_{1\leq i\leq n} $ be real valued random variables such that
\begin{align*}
&\mathbb E [Y_i| \mathcal F _{i-1}] = 0,\quad\text{for all} ~i=1,\ldots, n,
\end{align*}
then, for all $t\geq 0$ and ${v},{m}>0$, 
\begin{align*}
\mathbb P\Bigg(\Big|\sum_{i=1}^nY_i\Big|\geq t, \,\max_{i=1,\ldots, n}|Y_i |\leq {m},\, 
\sum_{i=1}^n &\mathbb E[Y_i^2|\mathcal F _{i-1}]\leq {v}\Bigg) \\
&\leq 2\exp\left(-\frac{t^2}{2({v}+t{m}/3)} \right).
\end{align*}
\end{theorem}

\begin{proof}

Let us recall the Bennett inequality for supermartingales as given Freedman in 1975:

\medskip
\hfill\begin{minipage}{\dimexpr\textwidth-1.5cm}
\begin{theorem}\label{th:freedman0}
(\cite[Theorem 4.1]{freedman:1975})
If $(X_i)_{1\leq i\leq n} $ be a sequence of real-valued random variables such that 
\begin{align*}
& X_i \leq 1~~{\rm a.s.}~~~~\text{and}~~~~ \mathbb E [X_i| \mathcal F _{i-1}] \le 0~~{\rm a.s.}
~~~~\text{for all} ~i=1,\ldots, n,
\end{align*}
then, for all $a\geq 0$ and $b>0$, 
\begin{align*}
&\mathbb P \Big(\sum_{i=1}^nX_i\geq a,\, \sum_{i=1}^ n \mathbb E[ X_i^2|\mathcal F _{i-1} ] 
\leq b \Big) \leq \exp\Big(-bh(a/b)\Big)\\
&h(u)=(1+u)\log(1+u)-u.
\end{align*}
\end{theorem}
\xdef\tpd{\the\prevdepth}
\end{minipage}
\medskip

Let us recall also the classical inequality allowing to switch from the Bennett inequality to
the Bernstein inequality (\cite{massart} p.38 or \cite{pollard} p.193):
\begin{align*}
&h(u)\ge \frac{u^2}{2(1+u/3)}.
\end{align*}
By the Jensen inequality, the variables $X_i=\min(Y_i/{m},1)$ satisfy the assumptions of Theorem~\ref{th:freedman0} and we get in particular, since $X_i^2\le Y_i^2/{m}^2$,
\begin{align*}
\mathbb P\bigg(\sum_{i=1}^nY_i\geq t, &\,\max_{i=1,\ldots, n}|Y_i |\leq {m},\, 
\sum_{i=1}^n \mathbb E[Y_i^2|\mathcal F _{i-1}]\leq {v}\bigg) \\
&~~~~\le
\mathbb P\bigg(\sum_{i=1}^nX_i\geq t/{m}, \,\sum_{i=1}^n \mathbb E[X_i^2|\mathcal F _{i-1}]\leq {v}/m^2\bigg) \\
&~~~\leq \exp\left(-\frac{t^2}{2({v}+t{m}/3)} \right).
\end{align*}
By the symmetry of the assumptions on $(Y_i)$, the same inequality holds 
true with $-Y_i$ instead of $Y_i$ and we get the stated bound.
\end{proof}

\begin{theorem}\label{th:freedmanx}
Let $(\Omega, \mathcal F , (\mathcal F _{i})_{i\geq 1}, \mathbb P)$ be a filtered space. Let $(Y_i)_{i\geq 1} $ be a sequence of real valued stochastic processes defined on 
$\mathbb R^d$, adapted to $(\mathcal F_i)_{i\geq 1}$, such that
for any $x\in\mathbb R^d$
\begin{align*}
&\mathbb E [Y_i(x)| \mathcal F _{i-1}] = 0,\quad \text{for all } i\geq 1 .
\end{align*}
Consider $\varepsilon>0$ and let  $(\tilde Y_i)_{\geq 1} $ be another $(\mathcal F_i)_{i\geq 1}$-adapted sequence 
of nonnegative stochastic processes defined on $\mathbb R^d$ such that for all $i\geq 1$ and $x\in\mathbb R^d$
\begin{align*}
&\sup_{\|y\|\le\varepsilon}|Y_i(x+y)-Y_i(x)|\le \tilde Y_i(x).
\end{align*}
Let $n\geq 1$ and assume that for some $A\ge 0$ and some set $\Omega_1\subset \Omega$, one has for all $\omega\in\Omega_1$
and $\|x\|\le A$
\begin{align}
& \max_{i=1,\ldots, n} |Y_i(x)|\le m\label{thmm}\\
&\sum_{i=1} ^ n \mathbb E\big[Y_i(x)^2|\mathcal F_{i-1}]\le v\\
&\sum_{i=1} ^ n \mathbb E\big[\tilde Y_i(x)|\mathcal F_{i-1}\big]\le \tau
\end{align}
then, for all $t\geq 0$,
\begin{align*}
&\mathbb P\Big(\sup_{\|x\|\le A}\big|\sum_{i=1} ^ n Y_i(x)\big|>t+\tau,\Omega_1\Big)\le 4(1+2A/\varepsilon)^d \exp\left(-\frac{t^2}{8(\tilde v+2mt/3)} \right).
\end{align*}
with $\tilde v=\max(v,2m\tau).$
\end{theorem}
\begin{proof}
Notice that in view of (\ref{thmm}), $\tilde Y_i$ can be replaced with 
$(2m)\wedge \tilde Y_i$, hence we can assume 
\begin{align*}
\forall\omega\in\Omega_1,~\forall \|x\|\le A~,~~\tilde Y_i(x)\le2m.
\end{align*}
Let $(x_k)_{k=1,\ldots N} $ be an $\varepsilon$-grid over 
 $\{\|x\|\leq A\}$, i.e.,  $\min_{k=1,\ldots N} \|x-x_k\|\leq \varepsilon$ if $\|x\|\leq A$.
 One can choose $N\le 1+2A/\varepsilon$ (\cite{temlyakov} Corollary 3.4) 
Then for any $x\in A$ and $\omega\in\Omega_1$:
\begin{align*}
\big|\sum_{i=1} ^ n  Y_i(x)\big|
&\le \sup_k\big| \sum_{i=1} ^ n  Y_i(x_k)\big|+\sup_k\sum_{i=1} ^ n  \tilde Y_i(x_k)\\
&\le \sup_k\big| \sum_{i=1} ^ n  Y_i(x_k)\big|+\sup_k\big|\sum_{i=1} ^ n  \tilde Y_i(x_k)-\mathbb E[\tilde Y_i(x_k)|\mathcal F_{i-1}]\big|+\tau
\end{align*}
hence
\begin{align*}
\mathbb P \Big(\sup_{\|x\|\le A}&\big| \sum_{i=1} ^ n  Y_i(x)\big|>t+\tau,\Omega_1\Big)
\le \sum_{k=1}^N\mathbb P\Big(\big|\sum_{i=1} ^ n  Y_i(x_k)\big|>\frac t2,\Omega_1\Big)\\
&~~~+\sum_k\mathbb P\Big(\big|\sum_{i=1} ^ n  \tilde Y_i(x_k)-\mathbb E[\tilde Y_i(x_k)|\mathcal F_{i-1}]\big|>\frac t2,\Omega_1\Big).
\end{align*}
We shall apply two times Theorem~\ref{th:freedman}. 
For the first term, we take $m$ and $v$ as given. 
For the second term, the uniform bound required in Theorem~\ref{th:freedman}  is obtained through
\begin{align*}
\max_{i= 1,\ldots, n} |\tilde Y_i(x_k)-\mathbb E[\tilde Y_i(x_k)|\mathcal F_{i-1}]|
\le2m
\end{align*}
because $0\le \tilde Y_i\le 2m$ on $\Omega_1$, 
and the quadratic variation bound follows from:
\begin{align*}
\sum_{i=1} ^ n  \mathbb E\Big[\big(\tilde Y_i(x_k)-\mathbb E[\tilde Y_i(x_k)|\mathcal F_{i-1}]\big)^2|\mathcal F_{i-1}\Big]
&\le \sum_{i=1} ^ n  \mathbb E\big[\tilde Y_i(x_k)^2|\mathcal F_{i-1}\big]\\
&\le 2m\sum_{i=1} ^ n  \mathbb E\big[\tilde Y_i(x_k)|\mathcal F_{i-1}\big]\\
&\le 2m\tau
\end{align*}
Finally we get
\begin{align*}
\mathbb P\Big(\sup_{\|x\|\le A}\big| \sum_{i=1} ^ n  Y_i(x)\big|>t+\tau,\Omega_1\Big)
&\le 2(1+2A/\varepsilon)^d\exp\left(-\frac{t^2}{8(v+mt/3)} \right)\\
& + 2(1+2A/\varepsilon)^d\exp\left(-\frac{t^2}{8(2m\tau+2mt/3)} \right)
\end{align*}
which implies the result.
\end{proof}

\begin{corollary}\label{th:freedmanZ} 
Let $(\Omega, \mathcal F , (\mathcal F _{i})_{i\geq 1}, \mathbb P)$ be a filtered space and $(X_i,w_i)_{i\geq 1}\subset \mathbb R^d \times \mathbb R_{\geq 0}$ be an $(\mathcal F_i)_{i\geq 1}$-adapted sequence of random variables such that 
for any positive function $\varphi$ and any $i\geq 1$,
\begin{align*}
\mathbb E\big[w_{i} \varphi(X_i)|\mathcal F_{i-1}\big]=\int \varphi(y)\mu(dy),
\end{align*} 
where $\mu$ is a probability measure on $\mathbb R^d$. Let $n\geq 1$ and assume that for some $\Omega_1\in  \mathcal F$ and for some function $M:\mathbb R^d \to  \mathbb R_{\geq 0}$,
\begin{align*}
\forall\omega\in\Omega_1, \quad  \forall i = 1,\ldots ,n,\quad  & w_i \le  M (X_i).
\end{align*} 
Let $K$ be a nonnegative bounded Lipschitz function on $\mathbb R^d$. Define
\begin{align*}
&Z (x) = \sum_{i=1}^n  \Big(w_{i} K_{h}(x-X_i ) -\int K_{h}(x-y )\mu(dy)\Big)
\end{align*} 
with $h>0$ and $K_h(x)=h^{-d}K(x/h)$ and set
\begin{align*}
&U_K=\sup_{x\in \mathbb R^d} K(x)\quad \text{and} \quad L_K=\sup_{x\in\mathbb R^d}\sup_{\|y\|>0} \frac{|K(x+y)-K(x)|}{\|y\|}.
\end{align*} 
Then, for all $t\geq 0$,
\begin{align*}
&\mathbb P\Big(\sup_{\|x\|\le A}|Z(x)|>t+\tau,\Omega_1\Big)
  \le 4(1+2A/\varepsilon)^d \exp\left(-\frac{t^2}{8(\tilde v+2mt/3)} \right)
\end{align*}
with 
\begin{align*}
&U_w =  \max_{x\in \mathbb R^d } M(x)\\
&m=U_K(1+U_w ) h^{-d}\\
&v= n  h^{-d}U_K  \int M(u) K_{h}(x-u ) \mu(du) \\
&\tau= 2L_K\varepsilon nh^{-d-1}  \\
&\tilde v=\max(v,2m\tau)
\end{align*} 

\end{corollary} 
\begin{proof}We apply Theorem~\ref{th:freedmanx} with
\begin{align*}
Y_i(x) 
&= w_i  K_{h}(x-X_i )-\mathbb E\big[w_i K_{h}(x-X_i)|\mathcal F_{i-1}\big]\\
&= w_{i}  K_{h}(x-X_i )-\int K_{h}(x-u )\mu(du)
\end{align*} 
It remains to estimate $m,v$ and $\tau$. For $m$, we notice that 
\begin{align*}
&|Y_i(x)|\leq  U_K(1+U_w) h^{-d}.
\end{align*} 
Concerning $v$ we have  
\begin{align*}
\mathbb E\big[Y_i(x)^2|\mathcal F_{i-1}\big]
&=\mathbb E\big[w_{i}^2 K_{h}(x-X_i )^2|\mathcal F_{i-1}\big]\\
&\le h^{-d}U_K \mathbb E\big[w_{i} M(X_i) K_{h}(x-X_i ) |\mathcal F_{i-1}\big] \\
&=  h^{-d}U_K  \int M(y) K_{h}(x-u ) \mu(du) .
\end{align*} 
The estimation of $\tau$ is obtained by noticing that
\begin{align*}
&\sup_{\|y\|\le\varepsilon} |Y_i (x+y ) - Y_i (x) | \\
&=\sup_{\|y\|\le\varepsilon}w_{i} |K_h(x+  y - X_i)-K_h(x - X_i)| \\
&~~~~+\int \sup_{\|y\|\le\varepsilon} |K_h(x+  y - u)-K_h(x - u)| \mu(du) \\
&\le  h^{-d}L_K\varepsilon h^{-1} (w_i +1),
\end{align*} 
Take $\tilde Y_i (x) =2 w_{i} h^{-d}L_K\varepsilon h^{-1} $ and verify that, since $E[w_{i}|\mathcal F_{i-1}\big]=1$,
\begin{align*}
\sum_{i=1} ^n \mathbb E[\tilde Y_i(x)|\mathcal F_{i-1}\big]
&= 2L_K\varepsilon nh^{-d-1} .\qedhere
\end{align*} 
\end{proof}

\paragraph{Acknowledgement} The authors are grateful to Pierre E. Jacob for some helpful comments on the use of the \texttt{PAWL} package.

\bibliographystyle{chicago}
\bibliography{revision_F.bbl}

\begin{thebibliography}{}

\bibitem[\protect\citeauthoryear{Aza\"{\i}s, Delyon, and Portier}{Aza\"{\i}s
  et~al.}{2018}]{azais+d+p:2018}
Aza\"{\i}s, R., B.~Delyon, and F.~Portier (2018).
\newblock Integral estimation based on {M}arkovian design.
\newblock {\em Adv. in Appl. Probab.\/}~{\em 50\/}(3), 833--857.

\bibitem[\protect\citeauthoryear{Bertail and Portier}{Bertail and
  Portier}{2019}]{bertail+p:2018}
Bertail, P. and F.~Portier (2019).
\newblock Rademacher complexity for {M}arkov chains: applications to kernel
  smoothing and {M}etropolis-{H}astings.
\newblock {\em Bernoulli\/}~{\em 25\/}(4B), 3912--3938.

\bibitem[\protect\citeauthoryear{Bornn, Jacob, Del~Moral, and Doucet}{Bornn
  et~al.}{2013}]{bornn+j+d+d:2013}
Bornn, L., P.~E. Jacob, P.~Del~Moral, and A.~Doucet (2013).
\newblock An adaptive interacting wang--landau algorithm for automatic density
  exploration.
\newblock {\em Journal of Computational and Graphical Statistics\/}~{\em
  22\/}(3), 749--773.

\bibitem[\protect\citeauthoryear{Boucheron, Lugosi, and Massart}{Boucheron
  et~al.}{2013}]{massart}
Boucheron, S., G.~Lugosi, and P.~Massart (2013).
\newblock {\em Concentration inequalities}.
\newblock Oxford University Press, Oxford.
\newblock A nonasymptotic theory of independence, With a foreword by Michel
  Ledoux.

\bibitem[\protect\citeauthoryear{Capp{\'e}, Douc, Guillin, Marin, and
  Robert}{Capp{\'e} et~al.}{2008}]{cappe+d+g+m+r:2008}
Capp{\'e}, O., R.~Douc, A.~Guillin, J.-M. Marin, and C.~P. Robert (2008).
\newblock Adaptive importance sampling in general mixture classes.
\newblock {\em Statistics and Computing\/}~{\em 18\/}(4), 447--459.

\bibitem[\protect\citeauthoryear{Capp{\'e}, Guillin, Marin, and
  Robert}{Capp{\'e} et~al.}{2004}]{cappe+g+m+r:2004}
Capp{\'e}, O., A.~Guillin, J.-M. Marin, and C.~P. Robert (2004).
\newblock Population monte carlo.
\newblock {\em Journal of Computational and Graphical Statistics\/}~{\em
  13\/}(4), 907--929.

\bibitem[\protect\citeauthoryear{Chopin}{Chopin}{2004}]{chopin:2004}
Chopin, N. (2004).
\newblock Central limit theorem for sequential monte carlo methods and its
  application to bayesian inference.
\newblock {\em The Annals of Statistics\/}~{\em 32\/}(6), 2385--2411.

\bibitem[\protect\citeauthoryear{Del~Moral}{Del~Moral}{2013}]{del:2013}
Del~Moral, P. (2013).
\newblock {\em Mean field simulation for {M}onte {C}arlo integration}, Volume
  126 of {\em Monographs on Statistics and Applied Probability}.
\newblock CRC Press, Boca Raton, FL.

\bibitem[\protect\citeauthoryear{Delyon and Portier}{Delyon and
  Portier}{2018}]{delyon+p:2018}
Delyon, B. and F.~Portier (2018).
\newblock Asymptotic optimality of adaptive importance sampling.
\newblock In {\em Proceedings of the 32nd International Conference on Neural
  Information Processing Systems}, pp.\  3138--3148.

\bibitem[\protect\citeauthoryear{Douc, Guillin, Marin, and Robert}{Douc
  et~al.}{2007a}]{douc+g+m+r:2007a}
Douc, R., A.~Guillin, J.-M. Marin, and C.~P. Robert (2007a).
\newblock Convergence of adaptive mixtures of importance sampling schemes.
\newblock {\em The Annals of Statistics\/}, 420--448.

\bibitem[\protect\citeauthoryear{Douc, Guillin, Marin, and Robert}{Douc
  et~al.}{2007b}]{douc+g+m+r:2007b}
Douc, R., A.~Guillin, J.-M. Marin, and C.~P. Robert (2007b).
\newblock Minimum variance importance sampling via population monte carlo.
\newblock {\em ESAIM: Probability and Statistics\/}~{\em 11}, 427--447.

\bibitem[\protect\citeauthoryear{Elvira, Martino, Luengo, and Bugallo}{Elvira
  et~al.}{2019}]{elvira+m+l+b:2015}
Elvira, V., L.~Martino, D.~Luengo, and M.~F. Bugallo (2019).
\newblock Generalized multiple importance sampling.
\newblock {\em Statist. Sci.\/}~{\em 34\/}(1), 129--155.

\bibitem[\protect\citeauthoryear{Evans and Swartz}{Evans and
  Swartz}{2000}]{evans:2000}
Evans, M. and T.~Swartz (2000).
\newblock {\em Approximating integrals via {M}onte {C}arlo and deterministic
  methods}.
\newblock Oxford Statistical Science Series. Oxford University Press, Oxford.

\bibitem[\protect\citeauthoryear{Feng, Maggiar, Staum, and W{\"a}chter}{Feng
  et~al.}{2018}]{feng+m+s+w:2018}
Feng, M.~B., A.~Maggiar, J.~Staum, and A.~W{\"a}chter (2018).
\newblock Uniform convergence of sample average approximation with adaptive
  multiple importance sampling.
\newblock In {\em 2018 Winter Simulation Conference (WSC)}, pp.\  1646--1657.
  IEEE.

\bibitem[\protect\citeauthoryear{Folland}{Folland}{2013}]{folland:2013}
Folland, G.~B. (2013).
\newblock {\em Real analysis: modern techniques and their applications}.
\newblock John Wiley \& Sons.

\bibitem[\protect\citeauthoryear{Fort, Jourdain, Kuhn, Leli{\`e}vre, and
  Stoltz}{Fort et~al.}{2015}]{fort+j+k+l+s:2015}
Fort, G., B.~Jourdain, E.~Kuhn, T.~Leli{\`e}vre, and G.~Stoltz (2015).
\newblock Convergence of the wang-landau algorithm.
\newblock {\em Mathematics of Computation\/}~{\em 84\/}(295), 2297--2327.

\bibitem[\protect\citeauthoryear{Freedman}{Freedman}{1975}]{freedman:1975}
Freedman, D.~A. (1975).
\newblock On tail probabilities for martingales.
\newblock {\em Ann. Probability\/}~{\em 3}, 100--118.

\bibitem[\protect\citeauthoryear{Geweke}{Geweke}{1989}]{geweke:1989}
Geweke, J. (1989).
\newblock Bayesian inference in econometric models using monte carlo
  integration.
\newblock {\em Econometrica: Journal of the Econometric Society\/}, 1317--1339.

\bibitem[\protect\citeauthoryear{Gin{\'e} and Guillou}{Gin{\'e} and
  Guillou}{2001}]{gigui2001}
Gin{\'e}, E. and A.~Guillou (2001).
\newblock On consistency of kernel density estimators for randomly censored
  data: rates holding uniformly over adaptive intervals.
\newblock {\em Ann. Inst. H. Poincar\'e Probab. Statist.\/}~{\em 37\/}(4),
  503--522.

\bibitem[\protect\citeauthoryear{Gin{\'e} and Guillou}{Gin{\'e} and
  Guillou}{2002}]{gine+g:02}
Gin{\'e}, E. and A.~Guillou (2002).
\newblock Rates of strong uniform consistency for multivariate kernel density
  estimators.
\newblock {\em Ann. Inst. H. Poincar\'e Probab. Statist.\/}~{\em 38\/}(6),
  907--921.
\newblock En l'honneur de J. Bretagnolle, D. Dacunha-Castelle, I. Ibragimov.

\bibitem[\protect\citeauthoryear{Givens and Raftery}{Givens and
  Raftery}{1996}]{givens+r:1996}
Givens, G.~H. and A.~E. Raftery (1996).
\newblock Local adaptive importance sampling for multivariate densities with
  strong nonlinear relationships.
\newblock {\em Journal of the American Statistical Association\/}~{\em
  91\/}(433), 132--141.

\bibitem[\protect\citeauthoryear{Gordon, Salmond, and Smith}{Gordon
  et~al.}{1993}]{gordon+s+s:1993}
Gordon, N.~J., D.~J. Salmond, and A.~F. Smith (1993).
\newblock Novel approach to nonlinear/non-gaussian bayesian state estimation.
\newblock In {\em IEE proceedings F (radar and signal processing)}, Volume 140,
  pp.\  107--113. IET.

\bibitem[\protect\citeauthoryear{Haario, Saksman, and Tamminen}{Haario
  et~al.}{2001}]{haario+s+t:2001}
Haario, H., E.~Saksman, and J.~Tamminen (2001).
\newblock An adaptive metropolis algorithm.
\newblock {\em Bernoulli\/}~{\em 7\/}(2), 223--242.

\bibitem[\protect\citeauthoryear{Hall and Heyde}{Hall and
  Heyde}{1980}]{hall+h:1980}
Hall, P. and C.~C. Heyde (1980).
\newblock {\em Martingale limit theory and its application}.
\newblock Academic Press, Inc. [Harcourt Brace Jovanovich, Publishers], New
  York-London.
\newblock Probability and Mathematical Statistics.

\bibitem[\protect\citeauthoryear{Hansen}{Hansen}{2008}]{hansen:2008}
Hansen, B.~E. (2008).
\newblock Uniform convergence rates for kernel estimation with dependent data.
\newblock {\em Econometric Theory\/}~{\em 24\/}(3), 726--748.

\bibitem[\protect\citeauthoryear{Hesterberg}{Hesterberg}{1995}]{hesterberg:1995}
Hesterberg, T. (1995).
\newblock Weighted average importance sampling and defensive mixture
  distributions.
\newblock {\em Technometrics\/}~{\em 37\/}(2), 185--194.

\bibitem[\protect\citeauthoryear{Kloek and Van~Dijk}{Kloek and
  Van~Dijk}{1978}]{kloek+v:1978}
Kloek, T. and H.~K. Van~Dijk (1978).
\newblock Bayesian estimates of equation system parameters: an application of
  integration by monte carlo.
\newblock {\em Econometrica: Journal of the Econometric Society\/}, 1--19.

\bibitem[\protect\citeauthoryear{Marin, Pudlo, and Sedki}{Marin
  et~al.}{2019}]{marin+p+s:2019}
Marin, J.-M., P.~Pudlo, and M.~Sedki (2019).
\newblock Consistency of adaptive importance sampling and recycling schemes.
\newblock {\em Bernoulli\/}~{\em 25\/}(3), 1977--1998.

\bibitem[\protect\citeauthoryear{Neddermeyer}{Neddermeyer}{2009}]{neddermeyer:2009}
Neddermeyer, J.~C. (2009).
\newblock Computationally efficient nonparametric importance sampling.
\newblock {\em Journal of the American Statistical Association\/}~{\em
  104\/}(486), 788--802.

\bibitem[\protect\citeauthoryear{Oh and Berger}{Oh and
  Berger}{1992}]{ho+b:1992}
Oh, M.-S. and J.~O. Berger (1992).
\newblock Adaptive importance sampling in {M}onte {C}arlo integration.
\newblock {\em J. Statist. Comput. Simulation\/}~{\em 41\/}(3-4), 143--168.

\bibitem[\protect\citeauthoryear{Owen}{Owen}{2013}]{owen:13}
Owen, A.~B. (2013).
\newblock {\em Monte Carlo theory, methods and examples}.

\bibitem[\protect\citeauthoryear{Owen and Zhou}{Owen and
  Zhou}{2000}]{owen+z:2000}
Owen, A.~B. and Y.~Zhou (2000).
\newblock Safe and effective importance sampling.
\newblock {\em J. Amer. Statist. Assoc.\/}~{\em 95\/}(449), 135--143.

\bibitem[\protect\citeauthoryear{Pollard}{Pollard}{1984}]{pollard}
Pollard, D. (1984).
\newblock {\em Convergence of stochastic processes}.
\newblock Springer Series in Statistics. Springer-Verlag, New York.

\bibitem[\protect\citeauthoryear{Robert and Casella}{Robert and
  Casella}{2004}]{robert:2004}
Robert, C.~P. and G.~Casella (2004).
\newblock {\em Monte {C}arlo statistical methods\/} (Second ed.).
\newblock Springer Texts in Statistics. Springer-Verlag, New York.

\bibitem[\protect\citeauthoryear{Silverman}{Silverman}{2018}]{silverman:2018}
Silverman, B.~W. (2018).
\newblock {\em Density estimation for statistics and data analysis}.
\newblock Routledge.

\bibitem[\protect\citeauthoryear{Stone}{Stone}{1980}]{stone:1980}
Stone, C.~J. (1980).
\newblock Optimal rates of convergence for nonparametric estimators.
\newblock {\em The annals of Statistics\/}, 1348--1360.

\bibitem[\protect\citeauthoryear{Temlyakov}{Temlyakov}{1998}]{temlyakov}
Temlyakov, V.~N. (1998).
\newblock The best m-term approximation and greedy algorithms.
\newblock {\em Advances in Computational Mathematics\/}~{\em 8\/}(3), 249--265.

\bibitem[\protect\citeauthoryear{van~der Vaart}{van~der
  Vaart}{1998}]{vandervaart:1998}
van~der Vaart, A.~W. (1998).
\newblock {\em Asymptotic statistics}, Volume~3 of {\em Cambridge Series in
  Statistical and Probabilistic Mathematics}.
\newblock Cambridge University Press, Cambridge.

\bibitem[\protect\citeauthoryear{Wang and Landau}{Wang and
  Landau}{2001}]{wang+l:2001}
Wang, F. and D.~Landau (2001).
\newblock Determining the density of states for classical statistical models: A
  random walk algorithm to produce a flat histogram.
\newblock {\em Physical Review E\/}~{\em 64\/}(5), 056101.

\bibitem[\protect\citeauthoryear{West}{West}{1993}]{west:1993}
West, M. (1993).
\newblock Approximating posterior distributions by mixtures.
\newblock {\em Journal of the Royal Statistical Society: Series B
  (Methodological)\/}~{\em 55\/}(2), 409--422.

\bibitem[\protect\citeauthoryear{Zhang}{Zhang}{1996}]{zhang:1996}
Zhang, P. (1996).
\newblock Nonparametric importance sampling.
\newblock {\em J. Amer. Statist. Assoc.\/}~{\em 91\/}(435), 1245--1253.

\end{thebibliography}

\begin{appendices}

\section{Asymptotic normality of $f_n$}\label{sec:weakcvfn}

\begin{theorem}[asymptotic normality of density estimate]\label{th:cvd}
Assume \ref{cond:f}, \ref{cond:f_tail}, \ref{cond:K} and work under policy \eqref{sampler_update}. If there exists $\epsilon >0$ such that  $a_n +  h_n ^4 \ll \lambda_n^{1+\epsilon}   \ll 1$ then for any $x\in\mathbb R^d$, as $n\to \infty$,
\begin{align*}
\big(nh_n^d\big)^{1/2}\left(f_n(x)-\tilde f_n(x) \right)\leadsto \mathcal N\Big(0,  f(x)  \int K^2\Big) .
\end{align*}
\end{theorem}

\begin{proof}
Apply Lemma \ref{lemma:initial_consistency} to obtain that (\ref{qfhyp}) is valid with $f_n$ in place of $f_n^\circ$. In particular \eqref{important_bound_2} is valid.
As in the beginning of the proof of \eqref{gnorm}, we have
\begin{align*}
f_{n}(x)-\tilde f_{n}(x)
  = \frac1{I_n(1)}\left\{\Big(I_n(g)-\tilde f_{n}(x) \Big)-\big(I_n(1)-1\big) \tilde f_{n}(x)  \right\},
\end{align*}
with $g(y) = K_{h_n}(x-y)$. Because $\tilde f_{n}(x) \leq U_f$ and $n^{1/2} (I_n(1)-1) = o_{\mathbb P} (1)$, that has been established in the proof of \eqref{gnorm}, it suffices to prove that
\begin{align*}
&\sqrt {h_n^d/n}\,Z_n(y)\leadsto \mathcal N\Big(0, f(x)\int K^2 \Big).
\end{align*}
We will apply Theorem \ref{th:hall} with 
\begin{align*}
w_{n,i}
&= \sqrt {h_n^d/n}\,\left\{ w_i K_{h_n}(x-X_i ) - \int f(y) K_{h_n} (x-y)\,\diff y \right\}\\
&= \sqrt {h_n^d/n}\,\left\{ \frac{\varphi_n(X_i)}{q_{i-1}(X_i)}-\int \varphi_n \right\},
\end{align*}
and $\varphi_n(y)=f(y)K_{h_n}(x-y)$. Equation (\ref{hh1}) is satisfied. We now show (\ref{hh2}) with $v_*=f(x)\int K^2$, or equivalently that
\begin{align*}
\frac{h_n^d}n\sum_{i=1}^n \left\{\int \frac{\varphi_n^2 }{ q_{i-1}} -\Big(\int \varphi_n\Big)^2 \right\} \to f(x)\int K^2  , \qquad \text{in probability.}
\end{align*}
Since  $\int\varphi_n\le \sup_{x\in \mathbb R^d}|f(x) |$, and since
\begin{align*}
&h_n^d\int \frac{\varphi_n^2(y)}{ f(y)}\diff y
=\int f(x+h_nu)K(u)^2\diff u
\to f(x)\int K^2
\end{align*}
it suffices to prove that
\begin{align}\label{f2g22}
\frac{1}n\sum_{i=1}^n \int 
  h_n^d\Big|\frac{\varphi_n^2(y)}{ q_{i-1}(y)}-\frac{\varphi_n^2(y) }{f(y)}  \Big|\diff y
  \to 0  , \qquad \text{in probability.}
\end{align}
But using (\ref{important_bound_2}) and the specific form of $\varphi_n$, we have
\begin{align*}
\frac{1}n\sum_{i=1}^n\int  h_n^d\Big|\frac{\varphi_n^2(y)}{ q_i(y)}-\frac{\varphi_n^2(y) }{f(y)}  \Big|\diff y
&\le
  \frac{C}n\sum_{i=1}^n\int h_n^d\lambda_i^\varepsilon K_{h_n}(x-y)^2\diff y\\
&\le
  \frac{C}  n \int K ^2 \sum_{i=1}^n\lambda_i^{\varepsilon}
\end{align*}
which tends to zero since $\lambda_i$ tend to zero.

Finally, we verify the Lindeberg condition (\ref{hh3}). We have to prove that
\begin{align*}
\frac{h_n^d}n\sum_{i=1}^n \int \left(\frac{\varphi_n(y) }{ q_{i-1}(y)} -\int\varphi_n\right)^2
q_{i-1}(y) \mathrm I_{A_{n,i}} (y)
\,\diff y \to 0~~~~ \text{in probability}
\end{align*}
with  $A_{ni}=\big\{y:|\varphi_n(y) / q_{i-1}(y)-\int\varphi_n| >\varepsilon \sqrt{nh_n^{-d}}\big\}$. 
But using  the fact that $\lambda_i$ is decreasing:
\begin{align*}
\sup_{i\le n}\frac{ \varphi_n(y) }{ q_i(y)}
\le   U_K c_1 h_n^{-d}\lambda_n^{-1}
=  U_K c_1\sqrt{nh_n^{-d}}\Big(\frac{a_n\lambda_n^{-1}}{\sqrt{\log n}}\Big)
\end{align*}
Since by assumption $a_n=O(\lambda_n)$, the parenthesized term tends to zero, and this implies,
$\int\varphi_n$ being bounded, that for $n$ large enough,
the sets $A_{ni}$, $1\le i\le n$ are all empty. The Lindeberg condition is thus satisfied.
\end{proof}

\section{The compact case}\label{sec:compact_case}

In this section we present a bound for the variance term $f_n-\tilde f_n$, which is analogous to the one of Theorem~\ref{lemma:true_rate}. The bound on the bias term  $\tilde f_n-f$
can be easily treated analogously to Lemma~\ref{lemma:bias}, under suitable assumptions.


\begin{enumerate}[label=(\text{H}\arabic*),resume=count_cond]
\item \label{cond:compact_case}The support of $f$, $S_f$, is compact and for all $x\in S_f$ 
we have $ L_f\leq  f(x) \leq U_f$. For all $x\in S_f$, ${q_0}(x) \geq L_{q_0}>0$.

In addition
\begin{align}\label{eqast}
\min_{x\in S_f}\min_{h\le h_1}(\mathrm I_{\{S_f\}} * K_h)(x)=C_{SK}>0.
\end{align}
\end{enumerate}

It is not difficult to prove that (\ref{eqast}) is satisfied if $S_f$ is convex
(since it is also bounded). The following event 
\begin{align*}
\Omega_L = \{\omega \,:\, \forall n\geq 1 \, \, \inf_{y\in S_f} q_n(y) \geq L\} ,
\end{align*}
will play an important role in the following. We state the key property related to $\Omega_L$ in the following lemma.

\begin{lemma}\label{lemma:proba_control_compact_case}

Under \ref{cond:f}, \ref{cond:K}, \ref{cond:compact_case}, if $(\lambda_n)_{n\geq 0}$ and $(h_n)_{n\geq 1}$ are positive decreasing sequences such that  $ \log(n) /   nh_n^d  \ll \lambda_n$, 
then with probability $1$,
\begin{align*}
\liminf_{n\to \infty} \inf_{x\in S_f}  q_{n}(x)  \geq   L_fC_{SK}.
\end{align*}
Moreover $\mathbb P ( \Omega_L) \to 1$ as $L\to 0$.
\end{lemma}
\begin{proof}

Recall (\ref{decomp_final}) and apply Lemma \ref{lemma:initial_variance} (using that $\log(n) /   nh_n^d  \ll \lambda_n$ implies that $\log(n) /   n  \ll \lambda_n$) to obtain that $M_n = o(n)$. Finally, applying again Lemma \ref{lemma:initial_variance}, and identity (\ref{df1}), 
we obtain that with probability $1$,
\begin{align*}
\sup_{y\in S_f} |  f_{n}(y) - \tilde f_{n}(y)  |   \to  0. 
\end{align*}
Now write,
\begin{align*}
 q_{n}(y) &\geq (1-\lambda_n) f_n\geq (1-\lambda_n)( \tilde f_{n}(y)  -  | f_{n}(y) - \tilde f_{n}(y) | ).
\end{align*}
Using (\ref{eqast}) gives that
\begin{align*}
 \inf_{y\in S_f}q_{n}(y) &\geq  (1-\lambda_n) L_fC_{SK}  -  \sup_{y\in S_f}  | f_{n}(y) - \tilde f_{n}(y) | .
\end{align*}
Taking the limit we obtain the first statement. By assumption, 
the variable $U_n= \inf_{y\in S_f} q_n(y)$ satisfies  $U_n\geq  \lambda_n L_{q_0}$; 
since in addition $\lim\inf_n U_n\ge L_fC_{SK}$, we have $\inf_n U_n>0$ almost surely.
Hence $\mathbb P(\Omega_L)\to 1$ as $L\to 0$.
\end{proof}

This result shows that the conditions on $\lambda_n$ are weakened in the compact case as $ a_n^{2-\delta} \ll \lambda_n  $ is replaced by $ a_n^2 \ll \lambda_n $.

\begin{theorem}[compact case]\label{th:consistency_kernel_compact}
Under \ref{cond:f}, \ref{cond:K}, \ref{cond:compact_case}, if $(\lambda_n)_{n\geq 0}$ and $(h_n)_{n\geq 1}$ are positive decreasing sequences such that  $ a_n^2 \ll \lambda_n  $, we have
almost surely
\begin{align}\label{compZ}
&\sup_{\|x\| \leq n^r } | Z_n(x)| = O  \left( \sqrt{ \frac{n\log(n)}{  h_n^d  }} \right)\\
&   M_n    =  O (\sqrt {n\log n }) .\label{compM}
\end{align} 
As a consequence,
\begin{align*}
\sup_{\|x\| \leq n^r } | f_{n}(y) - \tilde f_{n}(y)  | =  O_{\mathbb P}  \left( \sqrt{ \frac{  \log(n) }{ n  h_n^d  }} \right).
\end{align*}

\end{theorem}

\begin{proof}

Let us fix $L$ and apply Corollary~\ref{th:freedmanZ} with $\Omega_1=\Omega_L$
and  $\varepsilon=\varepsilon_n=h_n/\sqrt n$:
\begin{align*}
&\mathbb P\Big(\sup_{\|x\|\le n^r}|Z(x)|>t+\tau,\Omega_L\Big)
  \le 4(2n^r/\varepsilon)^d \exp\left(-\frac{t^2}{8(\tilde v+2mt/3)} \right)
\end{align*}
with 
\begin{align*}
&\mu(du)=f(u)du\\
&m=U_K(1+U_w h_n^{-d})\le U_K(1+U_fL^{-1} h_n^{-d})\le Ch_n^{-d}\\
&v=\sum_k h_n^{-d}U_K\int\frac{f(u)}{q_{k-1}(u)}  K_{h_n}(x-u )\mu(du)\le Cn h_n^{-d}\\
&\tau= C \varepsilon nh_n^{-1}\\
&\tilde v=\max(v,2m\tau)\le Cn h_n^{-d}
\end{align*} 
Taking $\varepsilon=\varepsilon_n=h_n/\sqrt n$, $t=t_n=\sqrt{\gamma n\log(n)h_n^{-d}}$
for some large $\gamma$,
we have $t_n\gg\tau$ and $\tilde v_n\gg mt_n$ (because $a_n\ll 1$), and the bound becomes
\begin{align*}
\mathbb P\Big(\sup_{\|x\|\le n^r}|Z(x)|>2t_n,\Omega_L\Big)
&  \le 4(1+2n^{r+1/2}/h_n)^d \exp\left(-C\frac{\gamma n\log(n)h_n^{-d}}{n h_n^{-d}} \right)\\
 & \le 4(1+2n^{r+1/2}/h_n)^d n^{-C\gamma}
\end{align*}
By the Borel-Cantelli lemma:
\begin{align*}
\mathbb P\Big(\varlimsup_n\sup_{\|x\|\le n^r}\frac{|Z(x)|}{t_n}>2,\Omega_L\Big).
=0
\end{align*}
Since this is true for any $L$, and $P(\Omega_L)\to 1$ as $L\to 0$
(cf. Lemma~\ref{lemma:proba_control_compact_case})
we have proved (\ref{compZ}).

For the second statement, we will use (\ref{df1}). Let us apply Theorem \ref{th:freedman} with 
\begin{align*}
Y_i =\frac{f(X_i)}{q_{i-1}(X_i)}   - \int  f(x) dx.
\end{align*}
On the set $\Omega_L$, we have (bound on the quadratic variation)
\begin{align*}
{v}&=\sum_{i=1}^n \mathbb E \left[\left(  \frac{f(X_i) }{q_{i-1}(X_i)}    - \int   f(x) dx \right)^2 \Big | \mathcal F_{i-1}\right] 
 \leq   \sum_{i=1}^n \int \frac{   f(y)^2  }{q_{i-1}(y)}  \,\diff y
  \leq  U_f L^{-1} n .
\end{align*}
Still on $\Omega_L$, a bound on the martingale increments is given as
\begin{align*}
& {m} = \max_{i=1,\ldots,n} \sup_{y\in \mathbb R^d} \left|\frac{ f(y) }{ q_{i-1}(y)}   - 1 \right|\leq  U_fL^{-1} +1 .
\end{align*}
Theorem \ref{th:freedman} implies that for any $t>0$
\begin{align*}
\mathbb P \left(| M_n| >t , \Omega_L \right) & \leq 2 \exp\left( -\frac{ Ct^2 }
{ n  +  t }\right).
\end{align*}
for some $C>0$ depending only on $(f,L)$. Choosing $t = \gamma \sqrt {n\log n}$ 
with $\gamma$ large enough, we get that
\begin{align*}
\sum_n \mathbb P \left(| M_n| > \gamma \sqrt {n\log n} , \Omega_L  \right) <+\infty.
\end{align*}
Since $\mathbb P ( \Omega_L) \to 1$ as $L\to 0$ (cf. Lemma~\ref{lemma:proba_control_compact_case}),
(\ref{compM}) is proved.

Using (\ref{decomp_final}), we now directly get the third statement.
\end{proof}

\end{appendices}

\end{document}